\documentclass[a4paper,11pt,leqno]{article}

\usepackage[latin1]{inputenc}
\usepackage[T1]{fontenc} 
\usepackage[english]{babel}
\usepackage{verbatim}
\usepackage{calligra}
\usepackage{enumerate}
\usepackage{dsfont}
\usepackage{amsfonts}
\usepackage{amsmath}
\usepackage{amsthm}
\usepackage{amssymb}
\usepackage{mathtools}
\usepackage{mathrsfs}
\usepackage{color}
\usepackage{graphicx}
\usepackage{bm}
\usepackage{srcltx}

\usepackage{graphicx}		
\usepackage[hypcap=false]{caption}

\theoremstyle{definition}
\newtheorem{defi}{Definition}[section]

\theoremstyle{plane}
\newtheorem{thm}[defi]{Theorem}
\newtheorem{prop}[defi]{Proposition}
\newtheorem{cor}[defi]{Corollary}
\newtheorem{lemma}[defi]{Lemma}

\newcommand{\tbf}{\textbf}

\newcommand{\tsl}{\textsl}

\newcommand{\mbb}{\mathbb}

\newcommand{\mc}{\mathcal}

\newcommand{\veps}{\varepsilon}

\newcommand{\what}{\widehat}
\newcommand{\wtilde}{\widetilde}
\newcommand{\vphi}{\varphi}
\newcommand{\vrho}{\varrho}

\newcommand{\g}{\gamma}

\newcommand{\s}{\sigma}
\renewcommand{\t}{\tau}
\newcommand{\z}{\zeta}

\newcommand{\de}{\delta}
\renewcommand{\o}{\omega}

\newcommand{\R}{\mathbb{R}}

\newcommand{\N}{\mathbb{N}}
\newcommand{\Z}{\mathbb{Z}}

\newcommand{\E}{\mbb{E}}
\renewcommand{\H}{\mbb{H}}

\renewcommand{\div}{{\rm div}\,}

\newcommand{\curl}{{\rm curl}\,}

\newcommand{\Id}{{\rm Id}\,}
\newcommand{\Supp}{{\rm Supp}\,}
\newcommand{\eff}{{\rm eff}}

\newcommand{\dx}{ \, {\rm d} x}
\newcommand{\dt}{ \, {\rm d} t}

\newcommand{\dd}{\,{\rm d}}

\allowdisplaybreaks

\def\d{\partial}
\def\div{{\rm div}\,}

\begin{document}

\title{\textsc{\Large{\textbf{Effective velocity and $L^\infty$-based well-posedness for incompressible fluids with odd viscosity
}}}}

\author{
\normalsize\textsl{Francesco Fanelli}$\,^{1a,b,c}\qquad$ and $\qquad$
\textsl{Alexis F. Vasseur}$\,^{2}$ \vspace{.5cm} \\
\footnotesize{$\,^{1a}\;$ \textsc{BCAM -- Basque Center for Applied Mathematics}} \\ 
{\footnotesize Alameda de Mazarredo 14, E-48009 Bilbao, Basque Country, SPAIN} \vspace{.2cm} \\
\footnotesize{$\,^{1b}\;$ \textsc{Ikerbasque -- Basque Foundation for Science}} \\  
{\footnotesize Plaza Euskadi 5, E-48009 Bilbao, Basque Country, SPAIN} \vspace{.2cm} \\
\footnotesize{$\,^{1c}\;$ \textsc{Universit\'e Claude Bernard Lyon 1}, {\it Institut Camille Jordan -- UMR 5208}} \\ 
{\footnotesize 43 blvd. du 11 novembre 1918, F-69622 Villeurbanne cedex, FRANCE} \vspace{.2cm} \\
\footnotesize{$\,^{2}\;$ \textsc{University of Texas at Austin}}, 
{\footnotesize \it Department of Mathematics} \\ 
{\footnotesize 2515 Speedway Stop C1200, Austin, TX 78712, USA} \vspace{.3cm} \\
\footnotesize{Email addresses: $\,^{1}\;$\ttfamily{ffanelli@bcamath.org}}, $\;$
\footnotesize{$\,^{2}\;$\ttfamily{vasseur@math.utexas.edu}}
\vspace{.2cm}
}

\date\today

\maketitle

\subsubsection*{Abstract}
{\footnotesize 
The present paper is concerned with the well-posedness theory for non-homogeneous incompressible fluids exhibiting odd (non-dissipative) viscosity effects.
Differently from previous works, we consider here the full odd viscosity tensor.
Similarly to the work of Bresch and Desjardins in compressible fluid mechanics,
we identify the presence of an effective velocity in the system, linking the velocity field
of the fluid and the gradient of a suitable function of the density. By use of this effective velocity, we propose a new formulation of the original system of
equations, thus highlighting a strong similarity with the equations of the ideal magnetohydrodynamics.
By taking advantage of the new formulation of the equations, 
we establish a local in time well-posedness theory in Besov spaces based on $L^\infty$
and prove a lower bound for the lifespan of the solutions implying ``asymptotically global'' existence: in the regime of small
initial density variations, $\rho_0-1= O(\veps)$ for small $\veps>0$, the corresponding solution is defined up to some time $T_\veps>0$ satisfying
the property $T_\veps\,\longrightarrow\,+\infty$ when $\veps\to0^+$.

}

\paragraph*{\small 2020 Mathematics Subject Classification:}{\footnotesize
35Q35 
\ (primary);
\ 35B65, 
\ 76B03, 
\ 76D09 
\ (secondary).}

\paragraph*{\small Keywords: }{\footnotesize incompressible fluids; odd viscosity; variable density; Els\"asser formulation; endpoint Besov spaces;
improved lifespan.}


\section{Introduction} \label{s:intro}

It is a common opinion that viscosity in fluids is a mechanism which dissipates energy. However, real life phenomena offer a rich variety of examples associated
to fluid mechanics where \emph{non-dissipative} viscosity effects appear, owing to time-reversal symmetry or parity breakings
\cite{ganeshan2017odd}, \cite{K-S-F-V}.
Typical examples arise in the context of quantum fluids (for instance, electron fluids, magnetised plasmas, chiral superfluids, quantum Hall fluids\ldots), but
there are examples also from classical hydrodynamics, like \tsl{e.g.} polyatomic gases, chiral active matter and vortex fluids.
As a consequence of the parity violation, the viscous stress tensor related to the fluid presents a skew-symmetric component, often dubbed
\emph{odd viscosity} or \emph{Hall viscosity}.

Even though examples of fluids having odd viscosity have been known for a long time in the context of superfluids and plasma physics \cite{Land-Lif_Fluids},
the subject gained renewed interest after the seminal work \cite{Avron} by Avron, who discovered symmetry violations in two-dimensional quantum Hall fluids (see also
\cite{avron1995viscosity}). The case of planar flows is particularly interesting, because in this circumstance
the parity breaking is compatible with the fluid isotropy.
Work \cite{Avron} opened the road to investigations seeking for parity breaking phenomena in other fluid systems. We refer to the introduction of \cite{F-GB-S}
for more details about this and for an overview of the physical literature.

\subsection{The system of equations} \label{ss:equations}
The present paper deals with the well-posedness theory of non-homogeneous incompressible fluids which present odd viscosity effects.
We will focus on a model discussed in \cite{A-C-G-M}, although in that paper the fluid was assumed to be compressible. The incompressibility condition
imposed hereafter plays a crucial role in the analysis.

Let us present the precise system of equations considered throughout this work. For simplicity, we assume that the 
fluid fills the full plane $\R^2$. By denoting by $u=u(t,x)\in\R^2$ the velocity field of the fluid and by $\rho=\rho(t,x)\geq0$
its density field, the model reads as follows:
\begin{equation}\label{eq:odd}
\left\{
\begin{aligned}
& \d_t\rho\, +\, \div\big(\rho\,u\big)\, =\,0 \\
& \d_t \big(\rho\, u\big)\, +\, \div\big(\rho\, u \otimes u\big)\, +\, \nabla \Pi\,  +\,
\nu_0\, \div\Big(\rho \ \left(\nabla u^\perp\,+\,\nabla^\perp u\right)\Big)\, =\, 0 \\
& \div u\, =\,0\,,
\end{aligned}
\right.
\end{equation}
where the scalar function $\Pi=\Pi(t,x)$ represents the pressure field. As usual, the last equation expresses the incompressibility of the fluid;
correspondingly, the pressure gradient $\nabla\Pi$ can be interpreted as a Lagrangian multiplier enforcing this constraint in the equation for $u$.
The coefficient $\nu_0\in\R\setminus\{0\}$ appearing the the momentum equation represents the kinematic odd viscosity coefficient. It turns out that
its precise value does not play any role in the theory, however we will keep track of it in our computations for the sake of precision.

Next, let us explain the notation used in equations \eqref{eq:odd}.
For a vector $v=(v_1,v_2)\in\R^2$, we have set $v^\perp\,=\,(-v_2,v_1)$ to be the vector obtained by $v$ after a rotation of angle $\pi/2$.
In addition, by $\nabla v$ we mean the transpose of the Jacobian matrix $Dv$ associated to $v$; in particular, the vector columns of the matrix $\nabla v$ are
the vectors $\nabla v_1$ and $\nabla v_2$. The writing $\nabla v^\perp$ has to be interpreted in the same way.
Rotating everything of angle $\pi/2$, we set $D^\perp$ to be the ``Jacobian matrix'' obtained by using the operator $\nabla^\perp=(-\d_2,\d_1)$ instead of $\nabla$
and $\nabla^\perp w\,:=\,^t(D^\perp w)$ to be its transpose. In particular, the $j$-th vector column of $\nabla^\perp w$ is exactly the vector $\nabla^\perp w_j$.

The tensor
\begin{equation} \label{eq:odd-term}
\nu_0\,\div\Big(\rho \ \left(\nabla u^\perp\,+\,\nabla^\perp u\right)\Big)\,=\,\nu_0\,\div\Big(\rho\,\big(\mc O_1\,+\,\mc O_2\big)\Big)\,.
\end{equation}
is precisely the odd viscosity tensor, where, in order to lighten the notation, we have defined
\[
\mc O_1\,:=\,\nabla u^\perp \qquad\qquad \mbox{ and }\qquad\qquad
\mc O_2\,:=\,\nabla^\perp u\,. 
\]
Remark that, differently from the classical viscosity tensor case, the odd viscosity term \eqref{eq:odd-term}
is orthogonal to $u$ with respect to the $L^2$ scalar product, as it can be easily seen after an integration by parts (see Subsection \ref{ss:L^p} for more details).
We deduce that the odd viscosity tensor is \emph{non-dissipative} and it does not contribute at all to the total energy balance of the system.

It follows from the above discussion that no regularising effect for $u$ can be deduced from the odd viscosity term.
On the other hand, this term
involves loss of derivatives in the (classical) energy estimates: let us give some details on this (we refer to \cite{F-GB-S} for a more in-depth discussion).
First of all, if we compute the derivatives in \eqref{eq:odd-term}, we find
\[
\div\Big(\rho \, \big(\mc O_1\,+\,\mc O_2\big) \Big)\,=\, 
\rho\,\Delta u^\perp\,+\,\nabla\rho\cdot\nabla u^\perp\,+\,\nabla\rho\cdot\nabla^\perp u\,.
\]
Now, even though the first term on the right-hand side is skew-symmetric at the first order, the other two terms consume one derivative both on $u$ and $\rho$,
thus precluding any attempt of closing \tsl{a priori} estimates in any (say) Sobolev space $H^s$, for $s>0$ large enough.
Another issue appears at the level of the pressure: restricting for a while to the homogeneous setting $\rho\equiv1$, one sees that,
although the study of \eqref{eq:odd} reduces to the classical $2$-D incompressible Euler equations (which are globally well-posed),
the hydrodynamic pressure gradient $\nabla\Pi$ loses two derivatives with respect to the regularity of the velocity field. As for density-dependent fluids
one cannot avoid the appearing of the pressure gradient in the analysis (see \tsl{e.g.} \cite{D_2010}, \cite{D-F}), this represents a major problem
when studying propagation of regularity.

\subsection{Previous results} \label{ss:previous}

Despite the recent increasing amount of physical literature devoted to parity breaking phenomena and to odd viscosity fluid systems,
one must notice that not so many rigorous results are known on the mathematical side.

To the best of our knowledge, the first work devoted to the mathematical theory of odd viscosity fluids is \cite{G-O}. There, the authors studied
the free-surface problem for homogeneous (hence incompressible) flows with odd viscosity. As the fluid was assumed to have constant density, the
above mentioned difficulties linked with derivative losses were not encoutered therein.

More recently, in \cite{F-GB-S} the authors considered the density-dependent (still incompressible) framework. Similarly to the present work,
they considered the system proposed in \cite{A-C-G-M}, namely equations \eqref{eq:odd} above, where however they neglected the term $\mc O_2\,=\,\nabla^\perp u$.
They proved local well-posedness of the related system of equations in sufficiently regular Sobolev spaces $H^s$, with $s>2$, together with an
unexpected smoothing effect for the density function (which reveals to belong to $H^{s+1}$) and an effect \tsl{\`a la Hoff}
for the pressure gradient (namely, despite $\nabla\Pi$ merely belongs to $H^{s-2}$, one has $\nabla\big(\Pi-\nu_0\rho\o\big)\in H^{s-1}$, with $\o=\curl(u)$
being the vorticity of the fluid).

It is worth mentioning that the semplification $\mc O_2\equiv0$ made in \cite{F-GB-S} do not remove the two issues related to the loss of derivatives
mentioned at the end of Subsection \ref{ss:equations}. To overcome them,
the strategy of \cite{F-GB-S} was based on the discovery of a system of \emph{good unknowns} related to equations \eqref{eq:odd} with $\mc O_2\equiv0$,
namely the quantities
\[
\o\,:=\,\curl(u)\,=\,\d_1u_2-\d_2u_1\,,\quad \eta\,:=\,\curl(\rho\,u)\,=\,\d_1\big(\rho\,u_2\big)-\d_2\big(\rho\,u_1\big)\,,\quad
\theta\,:=\,\eta-\nu_0\Delta\rho\,,
\]
which allowed the authors to propagate higher order norms of the solution $\big(\rho,u,\nabla\Pi\big)$.
The key point of the analysis of \cite{F-GB-S} was the remark that both $\o$ and $\theta$ solve transport equations by $u$, with
$H^{s-1}$ right-hand sides, a fact which made it possible to avoid the loss of derivatives and propagate the $H^s$ norm of $u$ and the $H^{s+1}$ norm of $\rho-1$.
In fact, to be more accurate, an additional obstruction appeared in the equation for the vorticity $\o$, owing to the presence of a
term linked with density and pressure, that is
\[
 \nabla^\perp\left(\frac{1}{\rho}\right)\cdot\nabla\Pi\,.
\]
Keeping in mind what pointed out above, it is easily seen that this term
cannot be better than $H^{s-2}$. At this point, the identification of another important quantity, related to the pressure but possessing more regularity,
saved the game. More precisely, after noticing that $\nabla\big(\Pi-\nu_0\rho\o\big)$ belongs to $H^{s-1}$ and writing
\begin{align*}
 \nabla^\perp\left(\frac{1}{\rho}\right)\cdot\nabla\Pi\,&=\,\nabla^\perp\left(\frac{1}{\rho}\right)\cdot\nabla\big(\Pi-\nu_0\rho\o\big)\,+\,
\nu_0\,\nabla^\perp\left(\frac{1}{\rho}\right)\cdot\nabla\big(\rho\o\big) \\
&=\,
\nabla^\perp\left(\frac{1}{\rho}\right)\cdot\nabla\big(\Pi-\nu_0\rho\o\big)\,-\,\nu_0\,\nabla^\perp\log\rho\cdot\nabla\o\,,
\end{align*}
one remarks that $\o$ indeed satisfies a transport equation with a $H^{s-1}$ right-hand side; however, the transport field is not the
original velocity field $u$, but rather the \emph{effective velocity}
\[
 V_\eff\,:=\,u\,-\,\nu_0\,\nabla^\perp\log\rho\,.
\]

Observe that the assumption $\mc O_2\equiv0$ was thought to be crucial in \cite{F-GB-S}. Indeed, by writing
$\div\big(\rho\, \mc O_2\big)\,=\,\nabla\rho\cdot\nabla^\perp u$ 
and computing the $\curl$ of the momentum equation (in order to find an equation for $\theta$), one finds
\[
\curl\left(\div\big(\rho\, \mc O_2\big)\right) \,=\,\nabla\rho\cdot\nabla^\perp\o\,+\,{\rm l.o.t.}\,.
\]
Hence, we see that a critical term appears, which consumes one derivative of $\o$. Thus, a new derivative loss is created, making the approach of \cite{F-GB-S} inconclusive.

\subsection{Contents of the paper} \label{ss:contents}

The main goal of the present paper is to extend the well-posedness theory of \cite{F-GB-S} to the full odd viscosity model \eqref{eq:odd}
and, at the same time, to propose a different approach, which has the advantage of better highlight the role played by the good unknowns
from \cite{F-GB-S}.

Let us delve into the latter point. Our approach is strongly inspired by Bresch and Desjardins formulation of compressible fluid mechanics \cite{B-D_2004},
based on the discovery of an \emph{effective velocity} hidden in the system of equations and of a new entropy balance (named \emph{BD entropy})
related to it. We refer to \tsl{e.g.} \cite{B-D-Lin}, \cite{B-D}, \cite{Mel-V}, \cite{B-D-Z} and \cite{V-Yu} for previous and more recent advances 
on the theory of compressible fluids using the BD entropy structure.
Interestingly, an effective velocity approach has found applications also in numerical studies (see \tsl{e.g.} \cite{Guer-Pop}),
in order to improve the stability of certain approximation schemes.

Notice that all the above mentioned works were concerned with compressible fluid systems. In this paper,
we are able to identify a suitable effective velocity in an \emph{incompressible} context, namely for the full system of equations
\eqref{eq:odd}. We call this effective velocity $W_\eff$, in order to mark the (slight) difference with $V_\eff$ mentioned above.
In particular, it turns out that $W_\eff$ is also divergence-free and that
\[
 u\,-\,W_\eff\,=\,2\,\nu_0\,\nabla^\perp\log\rho\,.
\]
Surprisingly, in Bresch and Desjardins work \cite{B-D_2004} about the shallow water system, the effective velocity has a very similar form, namely $u-\nabla\log\rho$;
the presence of the $\nabla^\perp$ operator instead of $\nabla$ here is consistent with the fact that we are now working
in an incompressible setting.
Taking advantage of a sort of BD entropy structure underlying equations \eqref{eq:odd}, we can compute an equation for $W_\eff$: it turns out
that the effective velocity is transported by $u$, up to a suitable modification of the pressure function
(in accordance with what found in \cite{F-GB-S}).

On the other hand, we can reformulate the momentum balance equation, namely the second equation appearing in \eqref{eq:odd}, in terms of the new velocity
field $W_\eff$. This is reminiscent of Brenner's formulation of fluid mechanics (see \cite{Brenner_1}, \cite{Brenner_2} and \cite{Feir-Vass}
for more details), which postulates
that the mass velocity (based on the classical notion of mass transport) and the volume velocity (associated to the motion of individual
fluid particles) are \emph{not} the same, but differs the one from the other for a multiple of the density gradient.
However, it is worth remarking that, in Brenner's approach, the effective velocity is introduced \tsl{ad hoc} in the equations, whereas
$W_\eff$ is an intrinsic quantity of system \eqref{eq:odd}; in this sense, our reformulation is closer to the BD
entropy structure previously highlighted.
Then, when reformulating the momentum equation, we find that, up to introduce the same modification of the pressure function as above,
the original velocity field $u$ is itself transported by $W_\eff$.

In passing, we observe that the final reformulated system (see system \eqref{eq:odd-Els} below) shares a striking similarity
with the Els\"asser formulation of a density-dependent version of the ideal MHD equations, see for instance \cite{Cobb-F} and references therein.

Taking advantage of this transport structure (which, however, was present also in \cite{F-GB-S}), we are naturally led to solve
system \eqref{eq:odd} in critical Besov spaces $B^s_{p,r}$ which are embedded in the set of globally Lipschitz functions.
By focusing only in the endpoint situation $p=+\infty$, we then establish a local well-posedness result in Besov spaces $B^s_{\infty,r}$, for any $s>1$ and
$r\in[1,+\infty]$ and for $s=r=1$.
Differently from the analysis of \cite{F-GB-S}, thanks to the ``Els\"asser formulation'' of the odd system,
we are able to do so without requiring any integrability on the density perturbations $\rho-1$,
namely here the density is only assumed to be bounded together with a high enough number of its derivatives.

In addition, in the present paper we investigate the lifespan of the constructed solutions.
The motivation comes from the remark that, when taking $\rho\equiv1$ in \eqref{eq:odd},
the equations degenerate to the classical incompressible Euler equations in $2$-D, which are known to be globally well-posed.
By resorting to improved transport estimates in Besov spaces $B^0_{p,r}$ of null regularity index due to Vishik \cite{Vis} (see also \cite{HK}),
we are able to make this observation quantitative and improve the classical lower bound for the lifespan of solutions which come from
quasi-linear hyperbolic theory. In particular, we prove that, for small variations of the density of size $\rho_0-1= O(\veps)$, for $\veps>0$ small,
the lifespan $T_\veps$ of the corresponding solution verifies the property $T_\veps\longrightarrow +\infty$ in the limit $\veps\to0^+$.
We observe that similar lower bounds have already been proved for other non-homogeneous perturbations of the incompressible Euler equations:
for the density-dependent counterpart \cite{D-F}, for the inviscid Boussinesq system \cite{D_2013},
for a zero-Mach number limit system \cite{F-L} and for the ideal MHD system \cite{Cobb-F} 
under an $L^2$ assumptions on both the velocity and magnetic fields.

\medbreak
Before moving on, let us give a brief overview of the paper.
First of all, in Section \ref{s:odd} we generalise the $H^s$ result of \cite{F-GB-S} to the full model \eqref{eq:odd}. This extension will be needed in the proof
of the well-posedness in endpoint Besov spaces.
In Section \ref{s:reform}, we investigate the role of the effective velocity $W_\eff$, we derive the ``Els\"asser formulation'' of the odd viscosity system and,
after that, we state the main results of the paper, namely a local in time well-posedness result in Besov spaces $B^s_{\infty,r}$ and an estimate
on the lifespan of the solutions. The rest of the paper is devoted to the proof of those results. In Section \ref{s:a-priori},
we obtain \tsl{a priori} estimates for (supposed to exist) smooth solutions of the system; a blow-up criterion will be derived
as a consequence of those estimates. Using the computations of the previous section, in
Section \ref{s:proof-inf} we carry out the rigorous proofs of existence, uniqueness and lower bound for the lifespan.
The paper is ended by an appendix, where we collect
some classical material from Fourier analysis and Littlewood-Paley theory that we need in our study.

\section*{Acknowledgements}

{\small

The work of the first author has been partially supported by the project CRISIS (ANR-20-CE40-0020-01), operated by the French National Research Agency (ANR). The work of the second author is partially supported by the NSF, Award Number: 2306852.

}

\subsection*{Notation} \label{s:not}

For the sake of clarity, let us recall here some notation we have introduced above and which will be freely used throughout this paper.
Here below, $\Omega$ denotes some open smooth domain in $\R^n$, with $n\geq 2$.

Given a function $\alpha:\Omega\longrightarrow\R$, we denote by $\nabla \alpha$ its gradient (with respect to the space variables, in the case where $\alpha$ depends also
on time).
Given a vector field $w$ on $\R^n$, defined on$\Omega\subset\R^n$, we denote by
$Dw$ its Jacobian matrix and by $\nabla w\,:=\,^t(Dw)$ the transpose of the Jacobian matrix.
Recall that the $j$-th vector column of $Dw$ is the vectors $\d_jw$; inversly, the $j$-th vector column of $\nabla w$ is given by the vector $\nabla w_j$
 
\smallbreak 
We now specialise on the two-dimensional case $n=2$.
Given a vector $w\in\R^2$, $w=(w_1,w_2)$, we denote by $w^\perp=(-w_2,w_1)$ the vector obtained from rotating $w$ of an angle $\pi/2$.
We define the $\curl$ operator in dimension $n=2$ as follows: for any vector field $w$ on $\R^2$, we set $\curl(w)\,:=\,\d_1w_2\,-\,\d_2w_1$.

We also set $D^\perp$ to be the ``Jacobian matrix'' obtained by using the operator $\nabla^\perp$ instead of $\nabla$; thus, the $j$-th column of the matrix
$D^\perp w$ is given by $-\d_2w$ if $j=1$, by $\d_1w$ if $j=2$. In other words, the $k$-th vector line of $D^\perp w$ is given by $\nabla^\perp w_k$.

Analogously, we define $\nabla^\perp w\,:=\,^t(D^\perp w)$. Notice that the $j$-th vector column of $\nabla^\perp w$ is exactly the vector $\nabla^\perp w_j$.


\section{Well-posedness for the full system in Sobolev spaces} \label{s:odd}

The goal of this section is to extend the theory developed in \cite{F-GB-S} for the case $\mc O_2\equiv0$ to the case of the full system \eqref{eq:odd}.

\begin{thm}\label{th:odd-full}
Let $s>2$. Take an initial datum $(\rho_0,u_0)\in L^\infty(\R^2)\times H^{s}(\R^2)$ and assume that
$$
\exists\,\rho_*>0\qquad \mbox{ such that }\qquad\qquad \rho_*\,\leq\,\rho_0(x)\,\leq\, \rho^*\,:=\,\left\|\rho_0\right\|_{L^\infty}\,.
$$
In addition, assume that $\rho_0-1\in H^{s+1}(\R^2)$.

Then, there exists $0<T=T(\rho_0,u_0)\leq +\infty$ and a unique solution $(\rho,u,\nabla \Pi)$ to system \eqref{eq:odd} on $[0,T]\times\R^2$ such that:
\begin{itemize}
 \item $\rho\in L^\infty\big([0,T]\times\R^2\big)$ verifies $\rho_*\leq\rho\leq\rho^* $ on $[0,T]\times\R^2$ and $\rho-1\in C\big([0,T];H^{s+1}(\R^2)\big)$;
 \item $u$ belongs to $C\big([0,T];H^s(\R^2)\big)$;
 \item $\nabla \Pi\in C\big([0,T];H^{s-2}(\R^2)\big)$, while the difference $\Pi-\nu_0\rho\o$, with $\o\,:=\,\curl(u)\,=\,\d_1u_2-\d_2u_1$ being
 the vorticity of $u$, satisfies the property $\nabla\left(\Pi\,-\,\nu_0\,\rho\,\o\right)\in C\big([0,T];H^{s-1}(\R^2)\big)$.
\end{itemize}
\end{thm}

\begin{proof}
Recall the notation $\mc O_1\,:=\,\nabla u^\perp$ and $\mc O_2\,:=\,\nabla^\perp u$ introduced above.
The proof of Theorem \ref{th:odd-full} relies on the following simple but fundamental observation: if we compute the difference $\mc O_2-\mc O_1$,
we get
\begin{equation} \label{eq:O-diff}
\mc O_2\,-\,\mc O_1\,=\,\nabla^\perp u\,-\,\nabla u^\perp\,=\,\o\,\Id\,,
\end{equation}
where $\Id$ is the $2\times2$ identity matrix and $\o\,:=\,\curl(u)\,=\,\d_1u^2-\d_2u^1$ is the vorticity of the fluid, as in the statement.
Therefore, computing the odd viscosity term, we find
\begin{align}
\label{eq:odd-comput}
\div\Big(\rho\,\big(\mc O_1\,+\,\mc O_2\big)\Big)\,&=\,\div\left(2\,\rho\,\mc O_1\right)\,+\,\div\Big(\rho\,\big(\mc O_2\,-\,\mc O_1\big)\Big) \\
\nonumber
&=\,2\,\div\left(\rho\,\nabla u^\perp\right)\,+\,\nabla\big(\rho\,\o\big)\,.
\end{align}

Thus, system \eqref{eq:odd} can be recasted in the form
\begin{equation}\label{eq:odd_reform}
\left\{
\begin{aligned}
& \d_t\rho\, +\, \div\big(\rho\,u\big)\, =\,0 \\
& \d_t \big(\rho\, u\big)\, +\, \div\big(\rho\, u \otimes u\big)\, +\, \nabla \wtilde\Pi\,  +\,
2\,\nu_0\, \div\Big(\rho \,\nabla u^\perp\Big)\, =\, 0 \\
& \div u\, =\,0\,,
\end{aligned}
\right.
\end{equation}
where we have defined
\[
\wtilde\Pi\,=\,\Pi\,+\,\nu_0\,\rho\,\o\,.
\]

At this point, we can apply the result of \cite{F-GB-S} to get, for some positive time $T>0$, the existence of a unique solution to system \eqref{eq:odd_reform}
related to the initial datum $\big(\rho_0,u_0\big)$ verifying the assumptions of the statement.
In order to stress that we are solving a different system, let us call
$\big(\vrho,U,\nabla\wtilde\Pi\big)$ such a solution. By Theorem 1.1 of \cite{F-GB-S}, we know the following properties:
\begin{itemize}
 \item $\vrho\in L^\infty\big([0,T]\times\R^2\big)$ verifies $\rho_*\leq\vrho\leq\rho^* $ pointwise on $[0,T]\times\R^2$,
 together with the regularity property $\vrho-1\in C\big([0,T];H^{s+1}(\R^2)\big)$;
 \item $U$ belongs to $C\big([0,T];H^s(\R^2)\big)$;
 \item $\nabla \wtilde\Pi\in C\big([0,T];H^{s-2}(\R^2)\big)$, together with $\nabla\left(\wtilde\Pi-2\nu_0\vrho\Omega\right)\in C\big([0,T];H^{s-1}(\R^2)\big)$,
 where $\Omega\,:=\,\curl(U)\,=\,\d_1U^2\,-\,\d_2U^1$ is the vorticity associated to the vector field $U$.
\end{itemize}

Now, let us define
\[
\rho\,:=\,\vrho\,,\qquad\quad u\,:=\,U\,,\qquad\quad
\Pi\,:=\,\wtilde\Pi\,-\,\nu_0\,\vrho\,\Omega\,. 
\]
Performing computations \eqref{eq:odd-comput} backwards, we see that the triplet $\big(\rho,u,\nabla\Pi\big)$ is a solution of system \eqref{eq:odd}.
In addition, it is plain to see that all the functions $\rho$, $u$ and $\nabla\Pi$ satisfy the claimed regularity properties.
Next, we notice that, by definition, we have
\[
\Pi\,-\,\nu_0\,\rho\,\o\,=\,\wtilde\Pi\,-\,2\,\nu_0\,\vrho\,\Omega\,.
\]
Owing to the fact that $\o=\Omega$, the previous relation immediately implies also the last property
$\nabla\big(\Pi\,-\,\nu_0\,\rho\,\o\big)\in C\big([0,T];H^{s-1}(\R^2)\big)$.
This completes the proof of the existence of a solution to the full system \eqref{eq:odd}.

\medbreak
The proof of uniqueness also uses the reformulation \eqref{eq:odd_reform}. As a matter of fact, assume to have two triplets
$\big(\rho_j,u_j,\nabla\Pi_j\big)_{j=1,2}$ of solutions to system \eqref{eq:odd} on some time interval $[0,T]$, both satisfying
the claimed regularity assumptions.

Then, each triplet must also solve \eqref{eq:odd_reform}, for suitable pressure functions $\wtilde\Pi_j\,=\,\Pi_j\,-\,\nu_0\,\rho_j\,\o_j$,
with $\o_j\,:=\,\curl(u_j)$ the vorticity of the velocity field $u_j$.
Then, the uniqueness result of \cite{F-GB-S} yields
$\rho_1=\rho_2$, $u_1=u_2$ and $\nabla\wtilde\Pi_1=\nabla\wtilde\Pi_2$. By using those properties, we find $\nabla\Pi_1=\nabla\Pi_2$. So, also
uniqueness of solutions is proved.

This completes the proof of Theorem \ref{th:odd-full}.
\end{proof}

\section{Reformulation and well-posedness in endpoint spaces} \label{s:reform}

In this section, we investigate the structure of equations \eqref{eq:odd} more in detail. In particular, we find a reformulation of the system,
which resembles to the Els\"asser formulation of the ideal MHD, see \tsl{e.g.} \cite{Cobb-F} for more details. Besides, such a reformulation
allows to better clarify the structure of the odd system highlighted in \cite{F-GB-S}.
Then, in Subsection \ref{ss:wp-inf} we state the main theorems of the paper, namely a well-posedness result in endpoint Besov spaces $B^s_{\infty,r}$,
together with an explicit lower bound, which improves the classical lower bound coming from hyperbolic theory, for the lifespan of the solutions.

\subsection{Role of the effective velocity} \label{ss:effective}

In the analysis of \cite{F-GB-S}, which were carried out for the reduced equations \eqref{eq:odd_reform}, the \emph{effective velocity}
\[
V_\eff\,:=\,u\,-\,\nu_0\,\nabla^\perp\log\rho
\]
played a fundamental role. Indeed, it turns out that $V_\eff$ was the vector field which advected the vorticity of the fluid. Such a phenomenon
allowed to avoid the loss of derivatives caused by the low regularity of the pressure term in the equations.

However, we remark that the introduction of that vector field was somehow \tsl{ad hoc}, as the presence of $V_\eff$ is hidden in the equations.
The goal of this subsection is to make the structure of the (full) odd system \eqref{eq:odd} more clear and the appearing of $V_\eff$ therein more evident.
As a matter of fact, it turns out that, for the study of the full odd model \eqref{eq:odd}, the effective velocity
\begin{equation} \label{def:W_eff}
 W_\eff\,:=\,u\,-\,2\,\nu_0\,\nabla^\perp\log\rho
\end{equation}
is a better suited vector field than $V_\eff$.

\subsubsection{BD-entropy structure} \label{sss:BD}

System \eqref{eq:odd} provides a sort of smoothing effect on the density function. As a matter of fact, from Theorem \ref{th:odd-full} we see
that $\rho-1$ has $H^{s+1}$ regularity, even though it is transported by a merely $H^s$ vector field.
The goal of the present section is to explain this smoothing effect by making a BD-entropy structure appear from the equations.

To this end, we start by differentiating the mass equation in \eqref{eq:odd}: from simple computations, we easily get
\begin{equation} \label{eq:Dlog}
\d_t\left(\nu_0\,\rho\,\nabla\log\rho\right)\,+\,\div\left(\nu_0\,\rho\,u\otimes\nabla\log\rho\right)\,+\,\div\left(\nu_0\,\rho\,Du\right)\,=\,0\,,
\end{equation}
where we have denoted
\[
\div\left(\nu_0\,\rho\,Du\right)\,=\,\sum_{j=1,2}\d_j\big(\nu_0\,\rho\,(Du)_{j,\bullet}\big)\,=\,\sum_{j=1,2}\d_j\big(\nu_0\,\rho\,\nabla u_j\big)\,.
\]
Next, we consider the momentum equation. The starting point is equation \eqref{eq:odd_reform},
namely
\[
\rho\,\d_tu\, +\, \rho\,u\cdot\nabla u\, +\, \nabla \wtilde\Pi\,  +\, 2\,\nu_0\, \div\Big(\rho \,\nabla u^\perp\Big)\, =\, 0\,,
\]
where we recall that $\wtilde \Pi$ is defined by the formula $\wtilde\Pi\,=\,\Pi\,+\,\nu_0\,\rho\,\o$.
By rotating the previous equation of angle $\pi/2$, we obtain
\begin{equation} \label{eq:u^perp}
\rho\,\d_tu^\perp\,+\,\rho\,u\cdot\nabla u^\perp\,+\,\nabla^\perp\wtilde\Pi\,-\,2\,\nu_0\,\div\left(\rho\,\nabla u\right)
\,=\,0\,,
\end{equation}
where we have used that $(v^\perp)^\perp=-v$. Observe that, without the reformulation used in \eqref{eq:odd_reform},
the $\mc O_2$ term would have given rise to an additional term of the kind 
$-2\nu_0\div\left(\rho\,Du\right)$, as $\nabla^\perp u^\perp\,=\,-Du$ owing to the divergence-free constraint over $u$.

At this point, we can write an equation for $W_\eff\,=\,u-2\nu_0\nabla^\perp\log\rho$, or rather for its rotation of angle $\pi/2$, namely the vector
\[
W_\eff^\perp\,=\,u^\perp\,+\,2\,\nu_0\,\nabla\log\rho\,. 
\]
As a matter of fact, taking the sum of twice equation \eqref{eq:Dlog} with equation \eqref{eq:u^perp}, we find
\begin{equation} \label{eq:W^perp}
\rho\,\d_tW_\eff^\perp\,+\,\rho\,u\cdot\nabla W_\eff^\perp\,+\,\nabla^\perp\wtilde\Pi\,+\,4\,\nu_0\,\div\left(\rho\,\mc A u\right)\,=\,0\,,
\end{equation}
where $\mc A u$ denotes the skew-symmetric part of the Jacobian matrix $Du$, namely
\[
\mc Au\,:=\,\frac{1}{2}\,\Big(Du\,-\,\nabla u\Big)\,.
\]
Observe that the following relation holds true:
\[
2\,\nu_0\,\div\left(\rho\,\mc Au\right)\,=\,-\,\nu_0\,\nabla^\perp\big(\rho\,\o\big)\,.
\]
Therefore, after performing another rotation of angle $\pi/2$, from \eqref{eq:W^perp} we deduce the equation
\begin{equation} \label{eq:W_eff}
\rho\,\d_tW_\eff\,+\,\rho\,u\cdot\nabla W_\eff\,+\,\nabla\Pi^0\,=\,0\,,
\end{equation}
where this time we have defined
\[
 \Pi^0\,:=\,\wtilde\Pi\,-\,2\,\nu_0\,\rho\,\o\,=\,\Pi\,-\,\nu_0\,\rho\,\o\,.
\]
Recall that $\o\,=\,\curl(u)$ is the vorticity associated to the velocity field $u$.

\subsubsection{Effective transport \tsl{\`a la Brenner}} \label{sss:Brenner}

In the previous paragraph, we have used the effective velocity $W_\eff$ as a \emph{passive} quantity,
inasmuch as we have shown that $W_\eff$ is advected by the velocity field $u$ of the fluid.
In this subsection we adopt a different point of view, which is inspired by the approach of Brenner (see \tsl{e.g.} \cite{Feir-Vass} and references therein)
and consists in looking at $W_\eff$ rather as an \emph{active} vector field which advects the quantities
of the dynamics.


\medbreak
To begin with, we observe that, thanks to the definition \eqref{def:W_eff} of the effective velocity $W_\eff$ and the
orthogonality relation $\nabla^\perp\log\rho\cdot\nabla\rho=0$,
we can recast the mass equation in the form
\begin{equation} \label{eq:mass_eff_W}
\d_t\rho\,+\,W_\eff\cdot\nabla\rho\,=\,0\,.
\end{equation}
This formulation of the mass equation is of course equivalent to the first equation of system \eqref{eq:odd}, at least for smooth density functions $\rho$.

Next, we introduce the effective velocity $W_\eff$ in the transport term of the momentum equation. Simple computations allow us to get
\begin{equation} \label{eq:mom_eff_provv}
 \rho\,\d_tu\,+\,\rho\,W_\eff\cdot\nabla u\,+\,2\,\nu_0\,\rho\,\nabla^\perp\log\rho\cdot\nabla u\,+\,\nabla\Pi\,+\,
 \nu_0\, \div\Big(\rho \ \left(\nabla u^\perp\,+\,\nabla^\perp u\right)\Big)\, =\, 0\,.
\end{equation}
Recall that the divergence-free constraints over both $u$ and $W_\eff$ hold:
\[
\div u\,=\,\div W_\eff\,=\,0\,.
\]
Thus, we develop the derivative and compute
\[
2\,\nu_0\,\rho\,\nabla^\perp\log\rho\cdot\nabla u\,=\, 
-\,2\,\nu_0\,\nabla\rho\cdot\nabla^\perp u\,=\,
-\,2\,\nu_0\,\div\left(\rho\,\nabla^\perp u\right)\,=\,-\,2\,\nu_0\,\div\big(\rho\,\mc O_2\big)\,.
\]
Therefore, insterting this latter equation into \eqref{eq:mom_eff_provv} and using relation \eqref{eq:O-diff}, we discover that
\begin{equation} \label{eq:mom_eff_W}
\rho\,\d_tu\,+\,\rho\,W_\eff\cdot\nabla u\,+\,\nabla\Pi^0\,=\,0\,,
\end{equation}
where $\Pi^0\,:=\,\Pi\,-\,\nu_0\,\rho\,\o$ is exactly the modified pressure introduced in equation \eqref{eq:W_eff} above.

\subsubsection{The final reformulated system} \label{sss:final}

Collecting equations \eqref{eq:W_eff} and \eqref{eq:mom_eff_W} together, we finally find the reformulated version of the odd system \eqref{eq:odd}, that is
\begin{equation} \label{eq:odd-Els}
\left\{\begin{array}{l}
\d_t\rho\,+\,u\cdot\nabla\rho\,=\,0 \\[1ex]
\rho\,\d_t u\,+\,\rho\,W_\eff\cdot\nabla u\,+\,\nabla\Pi^0\,=\,0 \\[1ex]
\rho\,\d_tW_\eff\,+\,\rho\,u\cdot\nabla W_\eff\,+\,\nabla\Pi^0\,=\,0 \\[1ex]
\div u\,=\,\div W_\eff\,=\,0\,.
       \end{array}
\right.
\end{equation}
Notice that, owing to \eqref{eq:mass_eff_W}, the first equation could have be written in many different ways, for instance
\[
\d_t\rho\,+\,\frac{1}{2}\,\big(u\,+\,W_\eff\big)\cdot\nabla\rho\,=\,0\,.
\]
However, this observation will not really be useful for what follows, so we prefer to stick to the original mass conservation equation.

\medbreak
Before going on, we want to point out a striking similarity between the reformulated odd-viscosity system \eqref{eq:odd-Els}
and the Els\"asser formulation of the ideal MHD equations (see \tsl{e.g.} \cite{Cobb-F} and references therein).
Namely, forgetting for a while the presence of the variable density $\rho$, the vector fields $u$ and $W_\eff$ in \eqref{eq:odd-Els}
play the same role of the Els\"asser variables, inasmuch as they are divergence-free vector field and each of them is transported by the other one.
In this perspective, the ``magnetic field'' $b$ appearing in the ``ideal MHD system'' behind equations \eqref{eq:odd-Els}
is $b\,=\,\big(u-W_\eff\big)/2\,=\,\nu_0\,\nabla^\perp\log\rho$.

\subsection{Well-posedness in a $L^\infty$-based functional framework} \label{ss:wp-inf}

The similarity of equations \eqref{eq:odd-Els}, with the Els\"asser formulation of the ideal MHD system
prompts us to investigate the well-posedness of system \eqref{eq:odd} in critical Besov spaces $B^s_{\infty,r}$ under the conditions
\begin{equation} \label{eq:Lip-cond}
 s>1\quad \mbox{ and }\quad r\in[1,+\infty]\,,\qquad\qquad \mbox{ or }\qquad\qquad s=r=1\,.
\end{equation}
As is well-known, either of those conditions guarantees that $B^s_{\infty,r}$ embeds in the space $W^{1,\infty}$ of globally Lipschitz functions.

The first main result of this part is contained in the next statement, concerning existence and uniqueness of solutions in the above mentioned
functional framework.
Properties of the lifespan of such solutions are stated in Theorem \ref{t:lifespan} below.

\begin{thm} \label{th:wp-inf}
Let $(s,r)\in\R\times[1,+\infty]$ such that either $s>1$, or $s=1$ and $r=1$.
Take an initial datum $\big(\rho_0,u_0\big)\in L^\infty(\R^2)\times L^2(\R^2)$ such that the following conditions are satisfied:
\begin{itemize}
 \item there exists a real number $\rho_*>0$ such that 
\[
\forall\,x\in\R^2\,,\qquad\qquad \rho_*\,\leq\, \rho_0(x)\,\leq\, \rho^*\,:=\,\left\|\rho_0\right\|_{L^\infty} \,;
\]
\item the initial velocity field $u_0$ belongs to the space $B^s_{\infty,r}(\R^2)$ and verifies $\div u_0=0$, while the initial density function $\rho_0$ belongs to
$B^{s+1}_{\infty,r}$.
\end{itemize}

Then, there exist a time $T\,=\,T(\rho_0,u_0)\,\in\;]0,+\infty]$ and a unique solution $\big(\rho,u,\nabla\Pi\big)$ to system
\eqref{eq:odd} on $[0,T]\times\R^2$, related to the initial datum $\big(\rho_0,u_0\big)$, satisfying the following regularity properties:
\begin{enumerate}[(i)]
 \item $\rho\in L^\infty\big([0,T]\times\R^2\big)$ verifies $\rho_*\leq\rho\leq\rho^* $, together with $\rho\in C\big([0,T];B^{s+1}_{\infty,r}(\R^2)\big)$;
\item $u$ belongs to $C\big([0,T];L^2(\R^2)\cap B^s_{\infty,r}(\R^2)\big)$;
\item the pressure gradient verifies $\nabla \Pi\in C\big([0,T];B^{s-2}_{\infty,r}(\R^2)\big)$ and, if we define $\o\,:=\,\curl(u)\,=\,\d_1u_2\,-\,\d_2u_1$
the vorticity of the fluid,
$\nabla\big(\Pi\,-\,\nu_0\,\rho\,\o\big)\in C\big([0,T];L^2(\R^2)\cap B^{s}_{\infty,r}(\R^2)\big)$.
\end{enumerate}
In the case $r=+\infty$, the continuity in time is only weak, thus the spaces $C\big([0,T];B^{\s}_{\infty,r}(\R^2)\big)$
have to be replaced by $C_w\big([0,T];B^{\s}_{\infty,\infty}(\R^2)\big)$.

In addition, if we define $W_\eff\,:=\,u\,-\,\nu_0\,\nabla^\perp\log\rho$ and $\Pi^0\,:=\,\Pi\,-\,\nu_0\,\rho\,\o$, then the set
$\big(\rho,u,W_\eff,\nabla\Pi^0\big)$ solves the ``Els\"asser formulation'' \eqref{eq:odd-Els} of the odd viscosity system.
\end{thm}

As claimed above, we now investigate some properties related to the existence time of the solutions constructed in Theorem \ref{th:wp-inf}.
First of all, we establish a blow-up criterion, which simplifies the one proposed in \cite{F-GB-S}, even though it requires a control on the full Hessian
matrix of $\rho$ instead of $\Delta\rho$ only.
In addition, we establish an explicit lower bound for the lifespan, implying that, for $\rho_0\approx1$, the solutions tend to be defined on larger and larger
time intervals. However, we have to remark that, owing to a loss of one derivative appearing in the estimates
involved in the proof of the result, we are able to state such a lower bound only for more regular initial data.

\begin{thm} \label{t:lifespan}
Let $\big(\rho_0,u_0\big)\in L^\infty(\R^2)\times L^2(\R^2)$ be an initial datum
satisfying the assumptions fixed in Theorem \ref{th:wp-inf}:
\begin{itemize}
 \item there exists $\rho_*>0$ such that $\rho_0\geq\rho_*$ over $\R^2$;
 \item $u_0\in B^s_{\infty,1}$, with $\div u_0=0$, and $\rho_0\in B^{s+1}_{\infty,1}$,
\end{itemize}
with the couple $(s,r)$ verifying one of the two conditions appearing in \eqref{eq:Lip-cond}.

Let the triplet $\big(\rho,u,\nabla\Pi\big)$
be the related solution to system \eqref{eq:odd}, as given by Theorem \ref{th:wp-inf}.
Let $T^*>0$ be\footnote{Observe that, by uniqueness, the maximal time of existence is well-defined.} its lifespan.
If $T^*<+\infty$, then 
\begin{equation} \label{eq:blow-up}
\int^{T^*}_0\Big(\left\|\nabla u(t)\right\|^{5/2}_{L^\infty}\,+\,\left\|\nabla^2\rho(t)\right\|^{5/2}_{L^\infty}\Big)\,\dt\,=\,+\,\infty\,.
\end{equation}

In addition, assume that $s>2$ and $r\in[1,+\infty]$, or $s=2$ and $r=1$.
Denote by $\o_0\,:=\,\curl(u_0)$ the vorticity of $u_0$.
Then, there exists a constant $K>0$, depending only on\footnote{Remark that the quantities $\rho_*$, $\rho^*$ and $\|u_0\|_{L^2}$
are preserved by the flow, see estimates \eqref{est:rho-inf} and \eqref{est:u-L^2} below.}
the triplet $\left(\rho_*,\rho^*,\|u_0\|_{L^2}\right)$ but not on the higher regularity norms of the initial datum, such that
\begin{align*}
 T^*\,&\geq\,\frac{K}{1\,+\,\left\|\left(\o_0,\Delta\log\rho_0\right)\right\|_{B^1_{\infty,1}}} \\
&\qquad\qquad\qquad
\times\,
\left[\log\left(1\,+\,\frac{K}{\left(1\,+\,\left\|\left(\o_0,\Delta\log\rho_0\right)\right\|_{B^1_{\infty,1}}\right)^3}\,\bullet\,\right)\right]^{\bigcirc4}\left(\frac{1}{\left\|\nabla\rho_0\right\|_{B^1_{\infty,1}}}\right)\,,
\end{align*}
where we have denoted by $\big[\log(1+C\,\bullet\,)\big]^{\bigcirc n}$ the $n$-th iterated logarithm function $z\mapsto \log(1+Cz)$.
\end{thm}


The rest of the paper is devoted to the proof of the previous theorems.
In the next section we derive \tsl{a priori} estimates for (supposed to exist) smooth solutions to system \eqref{eq:odd} in our functional framework.
From those bounds, we will also infer a continuation criterion, which in particular implies the blow-up criterion \eqref{eq:blow-up}.
The rigorous proof of the existence and uniqueness of solutions in the considered functional framework is postponed to Section \ref{s:proof-inf}:
this will prove Theorem \ref{th:wp-inf}.
In the same section, we will also show how to obtain the claimed lower bound on the lifespan of the solutions, thus completing
the proof to Theorem \ref{t:lifespan}.

\section{\tsl{A priori} estimates} \label{s:a-priori}

Assume to have a smooth solution $\big(\rho,u,\nabla\Pi\big)$ to system \eqref{eq:odd}, defined on $\R_+\times\R^2$. Let us show \tsl{a priori} estimates
for this solution.
In a first time, in Subsection \ref{ss:L^p} we establish basic bounds on Lebesgue norms. In Subsections \ref{ss:high} and \ref{ss:est_non-lin}, instead,
we see how to control high regularity norms. In Subsection \ref{ss:end-est} we collect all the bounds and derive uniform estimates
on some (possibly small) time interval $[0,T]$.

\subsection{Propagation of $L^p$ norms} \label{ss:L^p}

Our first goal is to look at estimates which rely on energy conservation properties of the system and on properties of transport equations
by a divergence-free vector field.

We start by considering the mass equation in system \eqref{eq:odd}: from it and from the assumptions made on $\rho_0$, we immediately deduce that
\begin{equation} \label{est:rho-inf}
\forall\,(t,x)\in\R_+\times\R^2\,,\qquad\qquad 0\,<\,\rho_*\,\leq\,\rho(t,x)\,\leq\,\rho^*\,,
\end{equation}
where we recall that $\rho^*\,:=\,\sup\rho_0$ has been defined as the maximum value of $\rho_0$. In addition, still from the mass equation we see that any
sufficiently regular scalar function of the density $\Phi=\Phi(\rho)$ still verifies the transport equation
\[
 \d_t\Phi(\rho)\,+\,u\cdot\nabla\Phi(\rho)\,=\,0\,,
\]
related to the initial datum $\Phi(\rho)_{|t=0}\,=\,\Phi(\rho_0)$.
This relation in particular implies that the bounds \eqref{est:rho-inf} hold true also for the function $\Phi(\rho)$:
\begin{equation} \label{est:Phi-inf}
\forall\,(t,x)\in\R_+\times\R^2\,,\qquad\qquad \inf_{\R^2}\Phi\circ\rho_0\,\leq\,\Phi\big(\rho(t,x)\big)\,\leq\,\sup_{\R^2}\Phi\circ\rho_0\,,
\end{equation}

Next, we perform an energy estimate on the momentum equation. Thus, we multiply the second equation appearing in \eqref{eq:odd} by $u$ and integrate over $\R^2$.
Observe that the odd viscosity tensor is skew-symmetric with respect to the $L^2$ scalar product, namely one has
\[
\int_{\R^2}\,\div\Big(\rho\,\left(\nabla u^\perp\,+\,\nabla^\perp u\right)\Big)\cdot u\,\dx\,=\,0\,.
\]
Thus, standard computations, based on integration by parts and the use of the mass equation, together with the orthogonality condition
between $\nabla\Pi$ and $u$, yield
\[
 \frac{1}{2}\,\frac{\rm d}{\dt}\int_{\R^2}\rho\,\left|u\right|^2\,\dx\,=\,0\,.
\]
Integrating the previous relation in time and using \eqref{est:rho-inf},
we easily find
\begin{equation} \label{est:u-L^2}
\forall\,t\,\geq\,0\,,\qquad\qquad \left\|u(t)\right\|_{L^2}\,\lesssim\,\left\|u_0\right\|_{L^2}\,.
\end{equation}

This having been established, we turn our attention on the propagation of high regularity norms.
Observe that, by writing any divergence-free vector field $f$ as 
\begin{equation} \label{eq:split-l-h}
f=\Delta_{-1}f+(\Id-\Delta_{-1})f\,,
\end{equation}
we deduce the following property:
for any $s\in\R$ and any $(p,r)\in[1,+\infty]\times[1,+\infty]$, one has
\[
 \left\|f\right\|_{B^s_{\infty,r}}\,\lesssim\,\|f\|_{L^p}\,+\,\left\|\big(\Id-\Delta_{-1}\big)\curl(f)\right\|_{B^{s-1}_{\infty,r}}\,,
\]
where we have defined the $\curl$ operator in two dimensions as $\curl(v)\,=\,\d_1v_2\,-\,\d_2v_1$ for any vector field $v\in\R^2$.
In fact, the previous inequality remains true in any dimension $d\geq2$, up to take define $\curl(v)=Dv-\nabla v$.

The previous inequality applies of course to $f=u$ with $p=2$, recall estimate \eqref{est:u-L^2} above.
In the case of $f=W_\eff$, the argument is similar, as we can bound
\begin{align}
\label{est:W_eff-Besov}
\left\|W_\eff\right\|_{B^s_{\infty,r}}\,&\lesssim\,\left\|\Delta_{-1}\big(u\,-\,2\nu_0\nabla^\perp\log\rho\big)\right\|_{L^\infty}\,+\,
\left\|\big(\Id-\Delta_{-1}\big)\curl\big(W_\eff\big)\right\|_{B^{s-1}_{\infty,r}} \\
\nonumber
&\lesssim\,\left\|u\right\|_{L^2}\,+\,\left\|\log\rho\right\|_{L^\infty}\,+\,\left\|\curl\big(W_\eff\big)\right\|_{B^{s-1}_{\infty,r}}\,.
\end{align}

Therefore, in light of inequalities \eqref{est:u-L^2} and \eqref{est:Phi-inf}, for bounding the $B^s_{\infty,r}$ norm of the solution it is
enough to bound the $B^{s-1}_{\infty,r}$ norm of the high frequencies $\Id-\Delta_{-1}$ of the $\curl$.
Establishing those bounds is the goal of the next paragraph.

\subsection{High frequency analysis} \label{ss:high}

Here we study the propagation of high regularity norms of the solution. We use two main ingredients: first of all, we work with the Els\"asser
formulation \eqref{eq:odd-Els} of the odd viscosity system; in addition, as anticipated above, we pass to the vorticity formulation of those equations
and study propagation of the vorticities in the less regular space $B^{s-1}_{\infty,r}$.
The drawback of this approach is that, in the endpoint case $s=r=1$, the space $B^0_{\infty,1}$ is no more a Banach algebra, a fact which obviously complicates
the analysis of the non-linear terms.

\medbreak
To begin with, let us introduce the two vorticities
\[
\o\,:=\,\curl(u) 
\qquad\qquad \mbox{ and }\qquad\qquad
\z_\eff\,:=\,\curl\big(W_\eff\big)\,. 
\]
Now, owing to \eqref{est:rho-inf}, we can divide the second and third equations in \eqref{eq:odd-Els} by $\rho$; then, we apply the $\curl$ operator
and find
\[ 
\left\{ \begin{array}{l}
         \d_t\o\,+\,W_\eff\cdot\nabla \o\,=\,\nabla^\perp\left(\dfrac{1}{\rho}\right)\cdot\nabla\Pi^0\,+\,\mc L\big(\nabla u,\nabla W_\eff\big) \\[1ex]
         \d_t\z_\eff\,+\,u\cdot\nabla \z_\eff\,=\,\nabla^\perp\left(\dfrac{1}{\rho}\right)\cdot\nabla\Pi^0\,+\,\mc L\big(\nabla W_\eff,\nabla u\big)\,,
        \end{array}
\right.
\] 
where, for $f$ and $g$ two divergence-free vector fields, we have defined
\[
\mc L\big(\nabla f,\nabla g\big)\,:=\,\d_1f_1\,\big(\d_1g_{2}\,+\,\d_2g_{1}\big)\,+\,\d_2g_2\,\big(\d_1f_2\,+\,\d_2f_1\big)\,.
\]
Notice that, as both $f$ and $g$ are divergence-free vector fields, $\mc L$ is a skew-symmetric operator. Thus, we have
\begin{align*}
 \mc L\big(\nabla g,\nabla f\big)\,&=\,-\,\mc L\big(\nabla f,\nabla g\big) \qquad\qquad \mbox{ and }\qquad\qquad
\mc L\big(\nabla f,\nabla g\big)\,=\,\mc L\big(\nabla f,\nabla(g-f)\big)\,. 
\end{align*}
Using those relations in the previous system, straightforward computations yield
\begin{equation} \label{eq:syst-vort}
\left\{ \begin{array}{l}
         \d_t\o\,+\,W_\eff\cdot\nabla \o\,=\,\nabla^\perp\left(\dfrac{1}{\rho}\right)\cdot\nabla\Pi^0\,-\,2\,\nu_0\,\mc B\big(\nabla u,\nabla^2\log\rho\big) \\[1ex]
         \d_t\z_\eff\,+\,u\cdot\nabla \z_\eff\,=\,\nabla^\perp\left(\dfrac{1}{\rho}\right)\cdot\nabla\Pi^0\,+\,2\,\nu_0\,\mc B\big(\nabla u,\nabla^2\log\rho\big)\,,
        \end{array}
\right.
\end{equation}
where the bilinear operator $\mc B$ is exactly the same operator introduced in \cite{F-GB-S}, namely
\[
\mc B\big(\nabla u,\nabla^2\alpha\big)\,:=\,\d_1u_1\,\big(\d_1^2\alpha\,-\,\d_2^2\alpha\big)\,+\,\d_1\d_2\alpha\,\big(\d_1u_2\,+\,\d_2u_1\big)
\]
for any scalar function $\alpha$. Indeed, it immediately follows from the definitions that
\begin{align*}
\mc L\big(\nabla u,\nabla\nabla^\perp\alpha\big)\,=\,\d_1u_1\,\big(\d_1^2\alpha\,-\,\d_2^2\alpha\big)\,+\,\d_1\d_2\alpha\,\big(\d_1u_2\,+\,\d_2u_1\big)\,=\,
\mc B\big(\nabla u,\nabla^2\alpha\big)\,.
\end{align*}

Our goal is to estimate the $B^{s-1}_{\infty,r}$ norm of the two vorticity functions, hence the $L^\infty$ norms of $\Delta_j\o$ and $\Delta_j\z_\eff$,
for any $j\geq-1$.
Owing to the symmetry of equations \eqref{eq:syst-vort}, let us focus on the equation for $\o$ only: applying operator $\Delta_j$, we find
\begin{align*}
\Big(\d_t\,+\,W_\eff\cdot\nabla\Big)\Delta_j\o\,=\,\Delta_j\left(\nabla^\perp\left(\frac{1}{\rho}\right)\cdot\nabla\Pi^0\right)\,-\,
2\,\nu_0\,\Delta_j\mc B\big(\nabla u,\nabla^2\log\rho\big)\,+\,\left[W_\eff\cdot\nabla,\Delta_j\right]\o\,.
\end{align*}
We now perform a $L^\infty$ estimate on this equation, multiply the resulting expression by $2^{j(s-1)}$ and sum up over $j\geq-1$:
by making use of Lemma \ref{l:CommBCD}, we obtain, for any $t\geq0$, the estimate
\begin{align*}
\left\|\o(t)\right\|_{B^{s-1}_{\infty,r}}\,&\lesssim\,\left\|\o(0)\right\|_{B^{s-1}_{\infty,r}}\,+\,\int^t_0
\left(\left\|\nabla^\perp\left(\dfrac{1}{\rho}\right)\cdot\nabla\Pi^0\right\|_{B^{s-1}_{\infty,r}}\,+\,
\left\|\mc B\big(\nabla u,\nabla^2\log\rho\big)\right\|_{B^{s-1}_{\infty,r}}\right)\,\dd\t \\
&\qquad\qquad\qquad\qquad\qquad +\,\int^t_0\Big(\left\|\nabla W_\eff\right\|_{L^\infty}\,\left\|\o\right\|_{B^{s-1}_{\infty,r}}\,+\,
\left\|\nabla W_\eff\right\|_{B^{s-1}_{\infty,r}}\,\|\o\|_{L^\infty}\Big)\,\dd\t\,,
\end{align*}
where we have set $\o(0)\,:=\,\curl(u_0)$.
It goes without saying that a similar estimate holds also for $\z_\eff$: if we define $\z_\eff(0)\,:=\,\curl\big(W_{\eff}(0)\big)$, with
$W_\eff(0)\,=\,u_0-2\nu_0\nabla^\perp\log\rho_0$, we have
\begin{align*}
\left\|\z_\eff(t)\right\|_{B^{s-1}_{\infty,r}}\,&\lesssim\,\left\|\z_{\eff}(0)\right\|_{B^{s-1}_{\infty,r}}\,+\,\int^t_0
\left(\left\|\nabla^\perp\left(\dfrac{1}{\rho}\right)\cdot\nabla\Pi^0\right\|_{B^{s-1}_{\infty,r}}\,+\,
\left\|\mc B\big(\nabla u,\nabla^2\log\rho\big)\right\|_{B^{s-1}_{\infty,r}}\right)\,\dd\t \\
&\qquad\qquad\qquad\qquad\qquad +\,\int^t_0\Big(\left\|\nabla u\right\|_{L^\infty}\,\left\|\z_\eff\right\|_{B^{s-1}_{\infty,r}}\,+\,
\left\|\nabla u\right\|_{B^{s-1}_{\infty,r}}\,\|\z_\eff\|_{L^\infty}\Big)\,\dd\t\,.
\end{align*}

Define now, for any $t\geq0$, the energy of the solution as
\begin{equation} \label{def:E_s}
E_s(t)\,:=\,\left\|u(t)\right\|_{L^2}\,+\,\left\|\rho(t)\right\|_{L^\infty}\,+\,\left\|\o(t)\right\|_{B^{s-1}_{\infty,r}}\,+\,
\left\|\z_\eff(t)\right\|_{B^{s-1}_{\infty,r}}\,,
\end{equation}
where we agree that, for $t=0$, the quantity $E_s(0)$ is defined analogously, but in terms of the initial datum $\big(\rho_0,u_0\big)$.
Then, putting together \eqref{est:rho-inf}, \eqref{est:u-L^2} and the previous inequalities for the Besov norms of $\o$ and $\z$, we get,
for any $t\geq0$, the bound
\begin{align}
\label{est:E_first}
E_s(t)\,&\lesssim\,E_s(0)\,+\,\int^t_0\Big(\left\|\nabla u\right\|_{L^\infty}\,+\,\left\|\nabla^2\rho\right\|_{L^\infty}\Big)\,E_s(\t)\,\dd\t \\
\nonumber
&\qquad\qquad +\,\int^t_0
\left(\left\|\nabla^\perp\left(\dfrac{1}{\rho}\right)\cdot\nabla\Pi^0\right\|_{B^{s-1}_{\infty,r}}\,+\,
\left\|\mc B\big(\nabla u,\nabla^2\log\rho\big)\right\|_{B^{s-1}_{\infty,r}}\right)\,\dd\t\,,
\end{align}
where the (implicit) multiplicative constant may depend also on the value of the kinematic odd viscosity coefficient $\nu_0\neq0$.

\subsection{Estimates of the non-linear terms} \label{ss:est_non-lin}

It remains us to estimate the non-linear terms appearing in the last line of inequality \eqref{est:E_first}. We start by considering the pressure term,
whose analysis is more involved. After that, we pass to the bounds for the bilinear operator $\mc B$.

\subsubsection{Estimates for the pressure} \label{sss:pressure}

We start by bounding the $B^{s-1}_{\infty,r}$ norm of the term involving the modified pressure $\nabla\Pi^0$. 
In the case $s>1$, we can directly use Corollary \ref{c:tame} to infer
\begin{align}
\label{est:bilin-press}
\left\|\nabla^\perp\left(\dfrac{1}{\rho}\right)\cdot\nabla\Pi^0\right\|_{B^{s-1}_{\infty,r}}\,&\lesssim\,\left\|\nabla\left(\dfrac{1}{\rho}\right)\right\|_{B^{s-1}_{\infty,r}}\,\left\|\nabla\Pi^0\right\|_{L^\infty}\,+\,
\left\|\nabla\left(\dfrac{1}{\rho}\right)\right\|_{L^\infty}\,\left\|\nabla\Pi^0\right\|_{B^{s-1}_{\infty,r}} \\
\nonumber 
&\lesssim\,\left\|\nabla\rho\right\|_{B^{s-1}_{\infty,r}}\,\left\|\nabla\Pi^0\right\|_{L^\infty}\,+\,
\left\|\nabla\rho\right\|_{L^\infty}\,\left\|\nabla\Pi^0\right\|_{B^{s-1}_{\infty,r}}\,,
\end{align}
where we have used Lemma \ref{l:paralin} for passing from the first inequality to the second one. Taking advantage of the commutator structure
\[
\nabla^\perp\left(\frac{1}{\rho}\right)\cdot\nabla\Pi^0\,=\,\curl\left(\frac{1}{\rho}\,\nabla\Pi^0\right)
\]
and repeating the analysis from Subsection 4.2 of \cite{D-F}, we see that the same estimate holds true even in the endpoint case $s=r=1$.
Of course, the previous multiplicative constants may depend on the lower and upper bounds, $\rho_*$ and $\rho^*$ respectively, of the initial density $\rho_0$.

Observe that, owing to Lemma \ref{l:paralin} again, for any $\s\in\,[0,s]$, we have
\begin{align}
\label{est:Drho-W}
\left\|\nabla\rho\right\|_{B^{\s}_{\infty,r}}\,&\lesssim\,\left\|\nabla\log\rho\right\|_{B^{\s}_{\infty,r}}\,\lesssim\,
\left\|W_\eff\right\|_{B^{\s}_{\infty,r}}\,+\,\left\|u\right\|_{B^{\s}_{\infty,r}}\,\lesssim\,E_s\,.
\end{align}

\medbreak
We now focus on the control of the Besov norm $B^{s-1}_{\infty,r}$ of $\nabla\Pi^0$. In fact, we are rather going to bound its $B^s_{\infty,r}$ norm.
We proceed in several steps.

\paragraph*{Step 1: frequency splitting.}
We start by cutting into low and high frequencies the pressure gradient, according to formula \eqref{eq:split-l-h}, and write
\begin{equation} \label{est:press-split}
\left\|\nabla\Pi^0\right\|_{B^s_{\infty,r}}\,\lesssim\,\left\|\nabla\Pi^0\right\|_{L^2}\,+\,\left\|\Delta\Pi^0\right\|_{B^{s-1}_{\infty,r}}\,.
\end{equation}
At this point, observe that we may follow the computations in \cite{D-F} (see Paragraph 3.1.2 therein)
and get, for any $s\geq1$ and $r\in[1,+\infty]$, with $r=1$ if $s=1$,
the inequality
\[
\left\|\nabla\Pi^0\right\|_{B^s_{\infty,r}}\,\lesssim\,
\left(1\,+\,\left\|\nabla\rho\right\|^\g_{B^{s-1}_{\infty,r}}\right)\,\left\|\nabla\Pi^0\right\|_{L^2}\,+\,
\left\|\rho\,\div\big(u\cdot\nabla W_\eff\big)\right\|_{B^{s-1}_{\infty,r}}\,,
\]
for a suitable $\g\in\,]0,1[\,$, depending on $s\geq1$. Nonetheless, this bound is not really suitable for the derivation of a good
continuation criterion.
We need to state more precise estimates, in the same spirit of Lemma 3.4 of \cite{F-GB-S}, paying
attention now that our functional framework is not always subcritical (in particular, when $s=r=1$, we work in a space which is
no more a Banach algebra).

\paragraph*{Step 2: the $L^2$ norm.}
From the third equation in \eqref{eq:odd-Els}, we get an equation for $\nabla\Pi^0$:
\begin{equation} \label{eq:pressure}
-\,\div\left(\frac{1}{\rho}\,\nabla\Pi^0\right)\,=\,\div\Big(u\cdot\nabla W_\eff\Big)\,.
\end{equation}
Applying Lax-Milgram theorem (see Lemma 2 of \cite{D_2010}), from the previous equation we gather
\begin{align}
\label{est:press-L^2}
\left\|\nabla\Pi^0\right\|_{L^2}\,&\lesssim\,\left\|u\right\|_{L^2}\,\left\|\nabla W_\eff\right\|_{L^\infty}\,
\lesssim\,E_s\,,
\end{align}
where we have used \eqref{est:u-L^2}, \eqref{est:W_eff-Besov} and the embedding $B^{s-1}_{\infty,r}\hookrightarrow L^\infty$.

For later use, notice that 
\begin{equation} \label{est:DW_eff}
\left\|\nabla W_\eff\right\|_{L^\infty}\,\lesssim\, \left\|\nabla u\right\|_{L^\infty}\,+\,\left\|\nabla \rho\right\|_{L^\infty}^2\,+\,
\left\|\nabla^2\rho\right\|_{L^\infty}\,.
\end{equation}

\paragraph*{Step 3: the Besov norm of $\Delta\Pi^0$.}
Keeping \eqref{est:press-split}, we now need an equation for $\Delta\Pi^0$.
Starting from \eqref{eq:pressure} again, simple computations yield
\begin{equation} \label{eq:Delta-press}
-\,\Delta\Pi^0\,=\,-\,\nabla\log\rho\cdot\nabla\Pi^0\,+\,\rho\,\nabla u:\nabla W_\eff\,,
\end{equation}
where we have used the fact that
\begin{equation} \label{eq:div-u-W}
 \div\big(u\cdot\nabla W_\eff\big)\,=\,\nabla u:\nabla W_\eff\,,
\end{equation}
which holds true owing to the divergence-free condition over the effective velocity $W_\eff$.

At this point, using Bony's paraproduct decomposition and the previous identity \eqref{eq:div-u-W}, we get the bound
\begin{align*}
\left\|\rho\,\div\big(u\cdot\nabla W_\eff\big)\right\|_{B^{s-1}_{\infty,r}}\,&\lesssim\,\left\|\rho\right\|_{B^{s}_{\infty,r}}\,
\left\|\div\big(u\cdot\nabla W_\eff\big)\right\|_{L^\infty}\,+\,\left\|\rho\right\|_{L^\infty}\,\left\|\div\big(u\cdot\nabla W_\eff\big)\right\|_{B^{s-1}_{\infty,r}} \\
&\lesssim\,
\Big(\rho^*+\left\|\nabla\rho\right\|_{B^{s-1}_{\infty,r}}\Big)\,
\left\|\nabla u\right\|_{L^\infty}\,\left\|\nabla W_\eff\right\|_{L^\infty}\, +\,\left\|\div\big(u\cdot\nabla W_\eff\big)\right\|_{B^{s-1}_{\infty,r}}\,.
\end{align*}
Using the same relation \eqref{eq:div-u-W} together with Bony's decomposition, we see that
\[
\left\|\div\big(u\cdot\nabla W_\eff\big)\right\|_{B^{s-1}_{\infty,r}}\,\lesssim\,\left\|u\right\|_{B^s_{\infty,r}}\,\left\|\nabla W_\eff\right\|_{L^\infty}\,+\,
\left\|\nabla u\right\|_{L^\infty}\,\left\|W_\eff\right\|_{B^s_{\infty,r}}
\]
for any $s>1$ and $r\in[1,+\infty]$ and also (we refer to \cite{D-F} for details) in the endpoint case $s=r=1$.
On the other hand, arguing again as in \cite{D-F} (see Paragraph 3.1.2 therein), for any $s\geq1$ and $r\in[1,+\infty]$, with
the convention that $r=1$ when $s=1$, we can estimate
\begin{align*}
\left\|\nabla\log\rho\cdot\nabla\Pi^0\right\|_{B^{s-1}_{\infty,r}}\,\lesssim\,\left\|\nabla\rho\right\|_{L^\infty}\,\left\|\nabla\Pi^0\right\|_{B^{s-1+\de}_{\infty,r}}\,+\,
\left\|\nabla\log\rho\right\|_{B^{s-1}_{\infty,r}}\,\left\|\nabla\Pi^0\right\|_{L^\infty}\,,
\end{align*}
where $\de>0$ can be taken arbitrarily small.
Observe that, when $s>1$, one can directly take $\de=0$; in other words, the $\de>0$ is needed only for dealing with the critical case $s=r=1$.

Collecting all those inequalities and making use of estimate \eqref{est:Drho-W}, we finally get a bound for the Besov norm of $\Delta\Pi^0$: we have
\begin{align}
\label{est:Delta-Pi^0}
\left\|\Delta \Pi^0\right\|_{B^{s-1}_{\infty,r}}\,&\lesssim\,\left\|\nabla\rho\right\|_{L^\infty}\,\left\|\nabla\Pi^0\right\|_{B^{s-1+\de}_{\infty,r}} \\
\nonumber
&\qquad \,+\,E_s\,\Big(\left\|\nabla\Pi^0\right\|_{L^\infty}\,+\,\left\|\nabla u\right\|_{L^\infty}\,\left\|\nabla W_\eff\right\|_{L^\infty}\,+\,
\left\|\nabla u\right\|_{L^\infty}\,+\,\left\|\nabla W_\eff\right\|_{L^\infty}\Big)\,.
\end{align}

\paragraph*{Step 4: interpolation.}
It is clear that the term $\left\|\nabla\Pi^0\right\|_{B^{s-1+\de}_{\infty,r}}$ appearing in \eqref{est:Delta-Pi^0} is of lower order,
thus it can be absorbed on the left-hand side by interpolation. In this step, we make the interpolation argument precise,
also in order to optimise the exponent appearing on the term depending on $\nabla\rho$.

In dimension $d=2$, owing to Proposition \ref{p:embed}, one has the following chain of continuous embeddings: $L^2\,\hookrightarrow\,B^{-1}_{\infty,2}\,
\hookrightarrow\,B^{-1-\de}_{\infty,r}$ for any $\de>0$ and any $r\in[1,+\infty]$ (actually, $\de=0$ would be enough if $r\geq2$).
As a consequence, if $s>1$, since one has $B^{s-2}_{\infty,r}\,\hookrightarrow\,B^{-1}_{\infty,2}$,
we can use the first embedding $L^2\,\hookrightarrow\,B^{-1}_{\infty,2}$ to estimate
\[
\left\|f\right\|_{B^{s-1}_{\infty,r}}\,\lesssim\,\left\|f\right\|^{1/2}_{B^{s-2}_{\infty,r}}\,\left\|f\right\|^{1/2}_{B^{s}_{\infty,r}}
\lesssim\,\left\|f\right\|^{1/2}_{L^2}\,\left\|f\right\|^{1/2}_{B^{s}_{\infty,r}}\,.
\]
For later use, in this case we set $\alpha(s)=1/2$ and we observe that $1/\big(1-\alpha(s)\big)=2$.

In the case $s=1$, instead, we need to resort to the second embedding $L^2\,\hookrightarrow\,B^{-1-\de}_{\infty,r}$ and to use a more refined interpolation
argument. More precisely, for $\de>0$ arbitrarily small, we can write
\[
\left\|f\right\|_{B^{s-1+\de}_{\infty,r}}\,\lesssim\,\left\|f\right\|^{1-\alpha}_{B^{-1-\de}_{\infty,r}}\,\left\|f\right\|^\alpha_{B^{s}_{\infty,r}}
\lesssim\,\left\|f\right\|^{1-\alpha}_{L^2}\,\left\|f\right\|^\alpha_{B^{s}_{\infty,r}}\,,
\]
where $\alpha\,=\,\alpha(s,\de)\,=\,(s+2\de)/(s+1+\de)$.
We immediately observe that
\[
\frac{1}{1-\alpha(s,\de)}\,=\,\frac{s+1+\de}{1-\de}\,,
\]
which can be made as close to $s+1=2$ as one needs.

\paragraph*{Step 5: bound for the pressure gradient.}
This having been observed, we are ready to establish a first bound for the Besov norm of $\nabla\Pi^0$.
The starting point is given by the previously established inequalities \eqref{est:press-L^2} and \eqref{est:Delta-Pi^0}.
We recall that, in \eqref{est:Delta-Pi^0}, one has $\de=0$ if $s>1$, whereas $\de>0$ arbitrarily small will be fixed later in case $s=1$.

Inserting those bounds into \eqref{est:press-split} and using the previously established interpolation inequalities, we infer the estimate
\begin{align*} 
 \left\|\nabla\Pi^0\right\|_{L^2\cap B^s_{\infty,r}}\,&\lesssim\,E_s\,\Big(1\,+\,\left\|\nabla\rho\right\|_{L^\infty}^\g\,+\,\left\|\nabla\Pi^0\right\|_{L^\infty}\,+\,
\left\|\nabla u\right\|_{L^\infty}^2\,+\,\left\|\nabla W_\eff\right\|^2_{L^\infty}\Big) \\
&\lesssim\,
E_s\,\Big(1\,+\,\left\|\nabla\rho\right\|_{L^\infty}^\g\,+\,\left\|\nabla\Pi^0\right\|_{L^\infty}\,+\,\left\|\big(\nabla u,\nabla^2\rho\big)\right\|^2_{L^\infty}
\,+\,\left\|\nabla\rho\right\|^4_{L^\infty}\Big)\,,
\end{align*} 
where the exponent $\g$ is defined as $\g\,:=\,1/(1-\alpha)$ and where we have also used \eqref{est:DW_eff} in order to pass from
the first inequality to the second one.
Observing that $\g\leq4$ always (up to take $\de>0$ small enough when $s=1$), we finally deduce the bound
\begin{equation} \label{est:D-Pi^0-Besov}
 \left\|\nabla\Pi^0\right\|_{L^2\cap B^s_{\infty,r}}\,\lesssim\,
E_s\,\Big(1\,+\,\left\|\nabla\rho\right\|_{L^\infty}^4\,+\,\left\|\nabla\Pi^0\right\|_{L^\infty}\,+\,\left\|\big(\nabla u,\nabla^2\rho\big)\right\|^2_{L^\infty}\Big)\,.
\end{equation}
Observe that the implicit multiplicative constant only depends on $\rho_*$, $\rho^*$ and $\|u_0\|_{L^2}$.

\paragraph*{Step 6: $L^\infty$ control of the pressure.}
Our next goal is to remove the presence of the term $\left\|\nabla\Pi^0\right\|_{L^\infty}$ in the previous bound.
We start by writing, for a suitable $\alpha\in\,]0,1[\,$, the Gagliardo-type interpolation inequality
\begin{equation} \label{est:starting}
\left\|\nabla\Pi^0\right\|_{L^\infty}\,\lesssim\,\left\|\nabla\Pi^0\right\|_{L^2}^{1/2}\,\left\|\Delta\Pi^0\right\|_{L^\infty}^{1/2}\,
\lesssim\,\left\|\nabla\Pi^0\right\|_{L^2}\,+\,\left\|\Delta\Pi^0\right\|_{L^\infty}\,,
\end{equation}
which holds true in dimension $d=2$.
The proof of the previous inequality is based on a decomposition into low and high frequencies which is rather standard;
we refer to the Appendix of \cite{Cobb-F_2021} for similar results.

Next, we use equation \eqref{eq:Delta-press} to bound
\[
\left\|\Delta\Pi^0\right\|_{L^\infty}\,\lesssim\,\left\|\nabla\rho\right\|_{L^\infty}\,\left\|\nabla\Pi^0\right\|_{L^\infty}\,+\,
\left\|\nabla u\right\|_{L^\infty}\,\left\|\nabla W_\eff\right\|_{L^\infty}\,,
\]
where the implicit multiplicative constant depends also on $\rho_*$ and $\rho^*$.
Applying the interpolation inequality \eqref{est:starting} again, we gather
\[
\left\|\Delta\Pi^0\right\|_{L^\infty}\,\lesssim\,\left\|\nabla\rho\right\|^{2}_{L^\infty}\,\left\|\nabla\Pi^0\right\|_{L^2}\,+\,
\left\|\nabla u\right\|_{L^\infty}\,\left\|\nabla W_\eff\right\|_{L^\infty}\,.
\]
Thus, plugging the previous estimate into \eqref{est:starting} and using inequalities \eqref{est:press-L^2} and \eqref{est:u-L^2}, we obtain
\begin{align} \label{est:Pi_inf}
\left\|\nabla\Pi^0\right\|_{L^\infty}\,&\lesssim\,\Big(1\,+\,\left\|\nabla\rho\right\|^{2}_{L^\infty}\Big)\,\left\|\nabla\Pi^0\right\|_{L^2}\,+\,
\left\|\nabla u\right\|_{L^\infty}\,\left\|\nabla W_\eff\right\|_{L^\infty} \\
\nonumber
&\lesssim\,\Big(1\,+\,\left\|\nabla\rho\right\|^{2}_{L^\infty}\,+\,
\left\|\nabla u\right\|_{L^\infty}\Big)\,\left\|\nabla W_\eff\right\|_{L^\infty} \\
\nonumber
&\lesssim\,1\,+\,\left\|\nabla\rho\right\|^{4}_{L^\infty}\,+\,
\left\|\nabla u\right\|^2_{L^\infty}\,+\,\left\|\nabla^2\rho\right\|^2_{L^\infty}\,,
\end{align}
where we have also used \eqref{est:DW_eff}. 

\paragraph*{Step 7: final estimate.}
It is now time to plug estimate \eqref{est:D-Pi^0-Besov} into \eqref{est:bilin-press}: we find
\begin{align*}
&\left\|\nabla^\perp\left(\dfrac{1}{\rho}\right)\cdot\nabla\Pi^0\right\|_{B^{s-1}_{\infty,r}} \\
&\qquad\qquad \lesssim\,
E_s\,\bigg(\left\|\nabla\Pi^0\right\|_{L^\infty}\,+\,\left\|\nabla\rho\right\|_{L^\infty}\,
\Big(1\,+\,\left\|\nabla\rho\right\|_{L^\infty}^4\,+\,\left\|\nabla\Pi^0\right\|_{L^\infty}\,+\,\left\|\big(\nabla u,\nabla^2\rho\big)\right\|^2_{L^\infty}\Big)\bigg)\,.
\end{align*}
Using \eqref{est:Pi_inf} in the previous bound finally yields
\begin{align}
\label{est:press-term}
&\left\|\nabla^\perp\left(\dfrac{1}{\rho}\right)\cdot\nabla\Pi^0\right\|_{B^{s-1}_{\infty,r}} \\
\nonumber
&\qquad\qquad\qquad \lesssim\,
E_s\,\Big(1\,+\,\left\|\nabla\rho\right\|_{L^\infty}^5\,+\,
\left\|\big(\nabla u,\nabla^2\rho\big)\right\|^2_{L^\infty}\,+\,
\left\|\nabla\rho\right\|_{L^\infty}\,\left\|\big(\nabla u,\nabla^2\rho\big)\right\|^2_{L^\infty}
\Big) \\
\nonumber
&\qquad\qquad\qquad \lesssim\,
E_s\,\Big(1\,+\,\left\|\nabla\rho\right\|_{L^\infty}^5\,+\,\left\|\nabla u\right\|^{5/2}_{L^\infty}\,+\,\left\|\nabla^2\rho\right\|^{5/2}_{L^\infty}\Big)\,,
\end{align}
where we have also applied the Young inequality to pass from the first to the second inequality.

\subsubsection{Estimates for the bilinear term} \label{sss:bilinear}
Now, we focus on the estimate of the bilinear term
\[
\mc B\big(\nabla u,\nabla^2\log\rho\big)\,=\,\d_1u_1\,\big(\d_1^2\log\rho\,-\,\d_2^2\log\rho\big)\,+\,\d_1\d_2\log\rho\,\big(\d_1u_2\,+\,\d_2u_1\big)
\]
in the space $B^{s-1}_{\infty,r}$. As already observed, we have the equality
\[
\mc B\big(\nabla u,\nabla^2\log\rho\big)\,=\,\mc L\big(\nabla u,\nabla\nabla^\perp\log\rho\big)\,.
\]
Thus, we can use Lemma 5.6 of \cite{Cobb-F} to get, even in the endpoint case $s=r=1$, the bound
\begin{align*}
\left\|\mc B\big(\nabla u,\nabla^2\log\rho\big)\right\|_{B^{s-1}_{\infty,r}}\,&\lesssim\,
\left\|\nabla u\right\|_{L^\infty}\,\left\|\nabla^\perp\log\rho\right\|_{B^{s}_{\infty,r}}\,+\,
\left\|\nabla\nabla^\perp\log\rho\right\|_{L^\infty}\,\left\|u\right\|_{B^{s}_{\infty,r}} \\
&\lesssim\,
\left\|\nabla u\right\|_{L^\infty}\,\left\|\nabla^\perp\log\rho\right\|_{B^{s}_{\infty,r}}\,+\,
\Big(\left\|\nabla\rho\right\|^2_{L^\infty}\,+\,\left\|\nabla^2\rho\right\|_{L^\infty}\Big)\,\left\|u\right\|_{B^{s}_{\infty,r}}\,.
\end{align*}
By triangular inequality, we easily gather that $\left\|\nabla^\perp\log\rho\right\|_{B^{s}_{\infty,r}}\,\leq\,2\,E_s$, hence
we deduce
\begin{align}
\label{est:B}
\left\|\mc B\big(\nabla u,\nabla^2\log\rho\big)\right\|_{B^{s-1}_{\infty,r}}\,&\lesssim\,
\Big(\left\|\nabla u\right\|_{L^\infty}\,+\,\left\|\nabla\rho\right\|^2_{L^\infty}\,+\,\left\|\nabla^2\rho\right\|_{L^\infty}\Big)\,E_s\,.
\end{align}

\subsection{Closing the estimates} \label{ss:end-est}

At this point, we can insert inequalities \eqref{est:press-term} and \eqref{est:B} into \eqref{est:E_first}: we obtain
\begin{equation} \label{est:E_second}
\forall\,t\geq0\,,\qquad\qquad
E_s(t)\,\lesssim\,E_s(0)\,+\,\int^t_0A(\t)\,E_s(\t)\,\dd\t\,,
\end{equation}
where we have set
\[
A(t)\,:=\,1\,+\,\left\|\nabla\rho\right\|_{L^\infty}^5\,+\,\left\|\nabla u\right\|^{5/2}_{L^\infty}\,+\,\left\|\nabla^2\rho\right\|^{5/2}_{L^\infty}\,.
\]

Observe that, by using the interpolation inequality
\begin{equation} \label{est:interp-rho}
 \left\|\nabla\rho\right\|^2_{L^\infty}\,\lesssim\,\left\|\rho\right\|_{L^\infty}\,\left\|\nabla^2\rho\right\|_{L^\infty}\,=\,
\rho^*\,\left\|\nabla^2\rho\right\|_{L^\infty}\,,
\end{equation}
we can bound
\begin{align*}
 A(t)\,&\lesssim\,1\,+\,\left\|\nabla^2\rho\right\|_{L^\infty}^{5/2}\,+\,\left\|\nabla u\right\|_{L^\infty}^{5/2}\,.
\end{align*}
From inequality \eqref{est:E_second} and the previous estimate, a classical argument (omitted here for simplicity)
allows us to deduce the following continuation criterion.

\begin{lemma} \label{l:cont-crit}
Let $\big(\rho_0,u_0\big)$ be initial data satisfying the assumptions fixed in Theorem \ref{th:wp-inf}. Let $(\rho,u,\nabla\Pi)$
be a solution to system \eqref{eq:odd} emanating from $\big(\rho_0,u_0\big)$,
defined on $[0,T^*[\,\times\R^2$ (for a suitable time $T^*>0$) and satisfying the regularity conditions stated in Theorem \ref{th:wp-inf}.

Assume that
\[
\int^{T^*}_0\left(\left\|\nabla^2\rho(t)\right\|_{L^\infty}^{5/2}\,+\,\left\|\nabla u(t)\right\|_{L^\infty}^{5/2}\right)\,\dt\,<\,+\infty\,.
\]
Then, the solution $\big(\rho,u,\nabla\Pi\big)$ can be continued beyond the time $T^*$ into a solution of system \eqref{eq:odd} enjoying
the same regularity properties.
\end{lemma}
Notice that this continuation criterion relies also on uniqueness of solutions in the considered functional framework: this will be proved in Subsection
\ref{ss:stab-uniqueness} below.
Now, from Lemma \ref{l:cont-crit} it is easy to deduce the blow-up criterion claimed in Theorem \ref{t:lifespan}, see \eqref{eq:blow-up}.

\medbreak
Next, we use estimate \eqref{est:E_second} to bound the quantity $E_s$ uniformly in some small time interval $[0,T]$. For this, we follow a classical argument.
Let us define the time $T>0$ as
\[
 T\,:=\,\sup\left\{t\geq0\;\Big|\quad \int^t_0A(\t)\,\dd\t\,\leq\,\log2\right\}\,.
\]
As, in our functional framework, $A(t)$ is a continuous function of time and its integral equals $0$ when $t=0$, it is clear that one indeed has $T>0$.

Then, by definition, from inequality \eqref{est:E_second} and an application of Gr\"onwall lemma we infer that
\begin{equation} \label{est:E_unif}
 \forall\,t\in[0,T]\,,\qquad\qquad E_s(t)\,\leq\,K\,E_s(0)\,,
\end{equation}
for a suitable universal constant $K>0$, depending only on $s$, $\rho_*$, $\rho^*$ and $\|u_0\|_{L^2}$, but not on the solution, nor on the high regularity norm
of the initial datum.

As a byproduct of \eqref{est:E_unif}, using also the definition \eqref{def:E_s} of $E_s$ and \eqref{est:W_eff-Besov}, we get
\begin{equation*} 
\sup_{t\in[0,T]}\left(\left\|u(t)\right\|_{L^2\cap B^s_{\infty,r}}\,+\,\left\|W_\eff(t)\right\|_{B^s_{\infty,r}}\right)\,\leq\,C\,,
\end{equation*}
for a universal constant $C>0$, depending only on the norms of the initial datum. Recalling the definition of $W_\eff$, we gather
also that
\[
 \sup_{t\in[0,T]}\left\|\rho(t)\right\|_{B^{s+1}_{\infty,r}}\,\leq\,C\,,
\]
which implies, by virtue of inequalities \eqref{est:D-Pi^0-Besov} and \eqref{est:Pi_inf}, the bound
$\sup_{t\in[0,T]}\left\|\nabla\Pi^0\right\|_{L^2\cap B^s_{\infty,r}}\,\leq\,C$.
Using the definition of $\Pi^0\,:=\,\Pi\,-\,\nu_0\,\rho\,\omega$, with $\omega\,=\,\curl(u)$, we finally recover uniform bounds for the (original) hydrodynamic
pressure $\Pi$: more precisely, we get
\[
\sup_{t\in[0,T]}\left\|\nabla\Pi\right\|_{B^{s-2}_{\infty,r}}\,\leq\,C\,.
\]

Before concluding, we observe that, from estimate \eqref{est:E_unif} and the definition of the time $T$, we could already deduce a first lower bound for the lifespan
of the solution, of the type
\begin{equation} \label{est:time-rough}
T\,\gtrsim\,\frac{1}{\left\|\rho_0\right\|_{B^s_{\infty,r}}\,+\,\left\|u_0\right\|_{L^2\cap B^s_{\infty,r}}}\,.
\end{equation}
Notice that this is the classical bound one can deduce from quasilinear hyperbolic theory.
We refer to Subsection \ref{ss:lifespan} below for the proof of the lower bound claimed in Theorem \ref{t:lifespan}, which will improve
\eqref{est:time-rough}.

\section{Proof of the main results} \label{s:proof-inf}

In this section, we give the rigorous proof of the theorems stated in Subsection \ref{ss:wp-inf}. In Subsection \ref{ss:proof-ex}, we prove the existence of a solution.
The stability and uniqueness issues are dealt with in Subsection \ref{ss:stab-uniqueness}, thus completing the proof
to Theorem \ref{th:wp-inf}. Finally, in Subsection \ref{ss:lifespan} we bound from below the lifespan of the solutions:
together with Lemma \ref{l:cont-crit}, also the proof of Theorem \ref{t:lifespan} will be then completed.

\subsection{Proof of existence} \label{ss:proof-ex}

We carry out here the proof of the existence of a solution to system \eqref{eq:odd} at the given level of regularity, as claimed in Theorem \ref{th:wp-inf}.
We divide our proof in two different cases: in the first one, presented in Paragraph \ref{sss:L^2-dens}, we assume an additional $L^2$ condition
over the density variation $\rho_0-1$; in the second one, we treat the general case. Finally, in Paragraph \ref{sss:time-reg} we establish
the claimed time-continuity properties of the solution.

\subsubsection{Existence for $L^2$ density variation} \label{sss:L^2-dens}

In this paragraph, we establish existence of solutions under the additional assumption that $\rho_0-1$ belongs to $L^2(\R^2)$.
More precisely, we are going to prove the following lemma.

\begin{lemma} \label{l:ex-energy}
Let the couple $\big(\rho_0,u_0\big)$ be an initial datum for system \eqref{eq:odd}, which satisfies the assumptions fixed in Theorem \ref{th:wp-inf}.
Assume in addition that
\begin{equation} \label{hyp:dens-L^2}
 \rho_0-1\,\in\,L^2(\R^2)\,.
\end{equation}

Then, there exists a solution $\big(\rho,u,\nabla\Pi\big)$ to system \eqref{eq:odd}, defined on $[0,T]\times\R^2$ for some positive time $T>0$,
which verifies the following conditions:
\begin{itemize}
 \item $\rho\in L^\infty\big([0,T]\times\R^2\big)$ verifies $\rho_*\leq\rho\leq\rho^* $ and $\rho\in L^\infty\big([0,T];B^{s+1}_{\infty,r}\big)$;
\item $u$ belongs to $L^\infty\big([0,T];L^2\cap B^s_{\infty,r}\big)$;
\item the pressure gradient verifies $\nabla \Pi\in L^\infty\big([0,T];B^{s-2}_{\infty,r}\big)$, whereas the quantity
$\nabla\big(\Pi\,-\,\nu_0\,\rho\,\o\big)$ belongs to $L^\infty\big([0,T];L^2(\R^2)\cap B^{s}_{\infty,r}(\R^2)\big)$.
\end{itemize}
In addition, if we define $W_\eff\,:=\,u\,-\,2\,\nu_0\,\nabla^\perp\log\rho$, then $\big(\rho,u,W_\eff,\nabla\big(\Pi-\nu_0\rho\o\big)\big)$
solves the ``Els\"asser formulation'' \eqref{eq:odd-Els} of the odd viscosity system.
\end{lemma}

\begin{proof}
The proof follows a classical scheme. First of all, let us regularise the initial datum: for any $n\in\N$, we define
\[
\rho_0^n\,:=\,S_n\rho_0\qquad\qquad \mbox{ and }\qquad\qquad u_0^n\,:=\,S_nu_0\,,
\]
where $S_n$ is the low frequency cut-off operator defined in \eqref{eq:S_j} in the Appendix below. Observe that one has, for a suitable constant $C>0$
independent of $n\in\N$, the uniform bounds
\[
\forall\,n\in\N\,,\qquad\qquad \left\|\rho_0^n\right\|_{B^{s+1}_{\infty,r}}\,\leq\,C\,\left\|\rho_0\right\|_{B^{s+1}_{\infty,r}}
\qquad \mbox{ and }\qquad \left\|u_0^n\right\|_{L^2\cap B^s_{\infty,r}}\,\leq\,C\,\left\|u_0\right\|_{L^2\cap B^s_{\infty,r}}\,.
\]
In addition, one gets strong convergence properties of the approximate data to the original one. For simplicity of presentation,
from now on we assume that $r<+\infty$.
Then, it follows from Lemma 2.73 of \cite{BCD} that one has 
\[ 
\rho^n_0\,\longrightarrow\,\rho_0\quad \mbox{ in }\; B^{s+1}_{\infty,r}\,,\qquad\qquad
u^n_0\,\longrightarrow\,u_0\quad \mbox{ in }\; L^2\cap B^s_{\infty,r}\,.
\] 
in the limit $n\to+\infty$.
In the case $r=+\infty$, one has strong convergence only in the space $\bigcap_{\s<s}B^{\s}_{\infty,\infty}$, instead of $B^s_{\infty,\infty}$.

Remark that, by definition, one has $\rho^n_0-1\,=\,S_n\rho_0-1\,=\,S_n\left(\rho_0-1\right)$.
Therefore, owing to the finite energy assumptions $u_0\in L^2$ and $\rho_0-1\in L^2$, one immediately gathers that,
for all $n\in\N$, the approximate data $\rho^n_0$ and $u_0^n$ belong to the space
$H^\infty\,:=\,\bigcap_{\s\in\R}H^\s$. Hence, for any $n\in\N$ fixed, we can solve system \eqref{eq:odd} by using
Theorem \ref{th:odd-full}, finding a unique local in time solution $\big(\rho^n,u^n,\nabla\Pi^n\big)$ defined on some time interval $[0,T_n]$.
Notice that this solution is smooth in space, as it belongs to $H^\s$ for any $\s\geq0$.
Therefore we can follow the computations of Subsection \ref{ss:effective} and deduce that
$\big(\rho^n,u^n,W^n_\eff,\nabla\big(\Pi^n-\nu_0\rho^n\o^n\big)\big)$ solves \eqref{eq:odd-Els} for any $n\in\N$,
where we have defined $W^n_\eff\,=\,u^n\,-\,2\,\nu_0\,\nabla^\perp\log\rho^n$.
Thanks to this property, we repeat the analysis of Section \ref{s:a-priori} and deduce two important facts:
\begin{enumerate}[(i)]
 \item first of all, keeping estimate \eqref{est:time-rough} in mind, that
there exists an absolute time $T>0$, independent of $n\in\N$, such that $\inf_nT_n\geq T$;
 \item moreover, that the sequence of solutions $\big(\rho^n,u^n,\nabla\Pi^n\big)$ is uniformly bounded in the respective functional spaces:
\[
\sup_{n\in\N}\left(\left\|\rho^n\right\|_{L^\infty_T(B^{s+1}_{\infty,r})}\,+\,\left\|u^n\right\|_{L^\infty_T(L^2\cap B^s_{\infty,r})}\,+\,
\left\|\nabla\Pi^n\right\|_{L^\infty_T(B^{s-2}_{\infty,r})}\right)\,\leq\,C\,,
\]
for a universal constant $C>0$, depending only on the norms of the initial datum;
\item similarly, if we define $\omega^n\,:=\,\curl(u^n)$, we have that
\[
\sup_{n\in\N}\left\|\nabla\left(\Pi^n-\nu_0\,\rho^n\,\o^n\right)\right\|_{L^\infty_T(L^2\cap B^s_{\infty,r})}\,\leq\,C\,,
\]
where $C>0$ is as above.
\end{enumerate}

Using the uniform bounds mentioned in item (ii) above and the Fatou property of Besov spaces, it is easy to extract a subsequence (not relabelled here)
$\big(\rho^n,u^n,\nabla\Pi^n\big)_n$ converging weakly-$*$ (each component in the respective functional space) to a triplet $\big(\rho,u,\nabla\Pi\big)$.
A classical argument, based on the inspection of equations \eqref{eq:odd} and the use of Ascoli-Arzel\`a theorem,
allows also to deduce strong convergence properties in larger spaces; in turn, this implies that one can pass to the limit also
in the non-linear terms appearing in the weak formulation of the equations and prove that the triplet $\big(\rho,u,\nabla\Pi\big)$
is in fact the sought solution. We omit the details here.

It remains to check that $\big(\rho,u,W_\eff,\nabla\big(\Pi-\nu_0\rho\o\big)\big)$, where we have defined $W_\eff\,:=\,u\,-\,2\,\nu_0\,\nabla^\perp\log\rho$,
is a solution of the ``Els\"asser formulation'' \eqref{eq:odd-Els}. This property easily follows from a compactness argument similar to the one employed above,
after recalling that, for any $n\in\N$, the set $\big(\rho^n,u^n,W^n_\eff,\nabla\big(\Pi^n-\nu_0\rho^n\o^n\big)\big)$ also solves \eqref{eq:odd-Els}.
\end{proof}

\subsubsection{Proof of existence: the general case} \label{sss:general-dens}

In this paragraph, we are going to prove the existence of a solution at the claimed level of regularity in the general case, namely without the additional
assumption \eqref{hyp:dens-L^2}.
The leading idea is to reconduct ourselves to the setting considered in Lemma \ref{l:ex-energy}.
For this, we take a smooth compactly supported function $\eta\,\in\,C^\infty_0(\R^2)$ such that
\[
0\,\leq\,\eta\,\leq 1\,,\qquad\qquad \eta\equiv1 \; \mbox{ for } \; |x|\leq1\qquad \mbox{ and }\qquad \eta\equiv0 \; \mbox{ for } \; |x|\geq2\,,
\]
For any $k\in\N\setminus\{0\}$, we then set
\[
 \eta_k(x)\,:=\,\eta\left(\frac{x}{k}\right)\,.
\]

At this point, for any $k$ as above, we can define
\[
 r_0^k\,:=\,\eta_k\,\big(\rho_0\,-\,1\big)\,.
\]
Observe that one has
\begin{equation} \label{est:L^inf_r_0}
\rho_*\,-\,1\,\leq\,r^k_0\,\leq\,\rho^*\,-\,1\,.
\end{equation}
In particular, we see that $r_0^k$ is a bounded function with compact support, thus $r_0^k\in L^2$.
In addition, as $C^\infty_0\subset B^\s_{\infty,r}$ for any $\s\in\R$
and any $r\in[1,+\infty]$, by Corollary \ref{c:tame} we get that $r_0^k$ also belongs to $B^s_{\infty,r}$. Then, we can define the approximate densities
$\rho_0^k$ as
\[
 \rho_0^k\,:=\,1\,+\,r_0^k\,.
\]
It follows from the previous arguments and \eqref{est:L^inf_r_0} that
\begin{equation} \label{est:rho_0^k-inf}
\rho_*\,\leq\,\rho^k_0\,\leq\,\rho^*
\end{equation}
and that $\rho_0^k\,\in\,B^s_{\infty,r}$, together with the uniform bound
\begin{equation} \label{est:rho_0^k-B}
\left\|\rho^k_0\right\|_{B^s_{\infty,r}}\,\leq\,\left\|\rho^k_0\right\|_{L^\infty}\,+\,\left\|\nabla\rho^k_0\right\|_{B^{s-1}_{\infty,r}}\,
\leq\,\rho^*\,+\,\left\|r^k_0\right\|_{B^s_{\infty,r}}\,\leq\,C\,\left\|\rho_0\right\|_{B^s_{\infty,r}}\,,
\end{equation}
for a suitable universal constant $C>0$, not depending on $k\geq1$. Notice that, in stating the last inequality, we have used Corollary \ref{c:tame} again.
Finally, we point out that, by definition of $\eta_k$, one has $\rho_0^k\equiv\rho_0$ on the ball
$B(0,k)$ of center $0$ and radius $k$. 
Thus, we infer the convergence property 
\[ 
 \rho_0^k\,\longrightarrow\,\rho_0\qquad\qquad \mbox{ in }\quad L^\infty_{\rm loc}(\R^2)\,,
\] 
in the limit $k\to+\infty$.

\medbreak
This having been established, the rest of the proof follows the main lines of the proof of Lemma \ref{l:ex-energy}.
For any $k\geq1$ we can solve system \eqref{eq:odd} with initial datum $\big(\rho_0^k,u_0\big)$:
by applying Lemma \ref{l:ex-energy}, we gather the existence of a sequence $\big(\rho^k,u^k,\nabla\Pi^k\big)_{k\geq1}$ of solutions
to that system, defined on some time intervals $[0,T_k]$. Keep in mind that those quantities also solve the Els\"asser-type system \eqref{eq:odd-Els}.
By virtue of the uniform bounds \eqref{est:rho_0^k-inf} and \eqref{est:rho_0^k-B}
and of estimate \eqref{est:time-rough}, we deduce that
$T_k\geq T>0$, for a suitable time $T>0$ independent of $k$, together with the properties
\[
\rho_*\,\leq\,\rho^k\,\leq\,\rho^*\qquad \mbox{ and }\qquad
\sup_{k\geq1}\left(\left\|\rho^k\right\|_{L^\infty_T(B^{s+1}_{\infty,r})}\,+\,\left\|u^k\right\|_{L^\infty_T(L^2\cap B^s_{\infty,r})}
\right)\,\leq\,C\,,
\]
for a suitable constant $C>0$. Concerning the pressure functions, we get, after setting $\o^k\,:=\,\curl(u^k)$ to be the vorticity
of the velocity field $u^k$, the bounds
\[
\sup_{k\geq1}\left(\left\|\nabla\Pi^k\right\|_{L^\infty_T(B^{s-2}_{\infty,r})}\,+\,
\left\|\nabla\left(\Pi^k\,-\,\nu_0\,\rho^k\,\o^k\right)\right\|_{L^\infty_T(L^2\cap B^s_{\infty,r})}\right)\,\leq\,C\,.
\]

Thanks to those uniform bounds, we can mimick the compactness argument evoked at the end of the proof of Lemma \ref{l:ex-energy} and pass to the limit
in the weak formulation of both equations \eqref{eq:odd} and \eqref{eq:odd-Els}. In this way, we prove the existence of a solution
$\big(\rho,u,\nabla\Pi\big)$ on $[0,T]\times\R^2$, satisfying the claimed regularity properties and verifying both
systems  \eqref{eq:odd} and \eqref{eq:odd-Els}.

\subsubsection{Time regularity of the solution} \label{sss:time-reg}

What remains to do is to prove the claimed time regularity properties of the solution: this is the goal of the present paragraph.

Hence, assume to dispose of a triplet $\big(\rho,u,\nabla\Pi\big)$ defined on $[0,T]\times\R^2$, for a suitable time $T>0$, which solves equations
\eqref{eq:odd} and such that, after defining $W_\eff\,:=\,u-2\nu_0\nabla^\perp\log\rho$ and $\nabla\Pi^0\,:=\,\nabla\big(\Pi-\nu_0\rho\o\big)$,
with $\o=\curl(u)$, the set $\big(\rho,u,W_\eff,\nabla\Pi^0\big)$ satisfies equations \eqref{eq:odd-Els}. Finally, assume that:
\begin{itemize}
 \item there exist two constants $\rho_*<\rho^*$ such that $\rho_*\leq\rho\leq\rho^*$ and $\rho\in L^\infty_T\big(B^{s+1}_{\infty,r}\big)$;
\item both $u$ and $\nabla\Pi^0$ belong to $L^\infty_T\big(L^2\cap B^{s+1}_{\infty,r}\big)$;
\item the pressure gradient $\nabla\Pi$ belongs to $L^\infty_T\big(B^{s-2}_{\infty,r}\big)$.
\end{itemize}

We start by considering $u$, which solves the first equation appearing in system \eqref{eq:odd-Els}. Thus, $u$ is transported by a divergence-free
Lipschitz continuous vector field, namely $W_\eff\in L^\infty_T\big(L^2\cap B^{s}_{\infty,r}\big)$, with initial datum $u_0\in L^2\cap B^s_{\infty,r}$
and with an external force $\nabla\Pi^0/\rho$ which belong to $L^\infty_T\big(L^2\cap B^{s}_{\infty,r}\big)$.
Then, classical results on transport equations in Besov spaces (see \tsl{e.g.} Chapter 3 of \cite{BCD}) imply
that $u\in C\big([0,T];L^2\cap B^{s}_{\infty,r}\big)$, as sought.

By the same token, we deduce that both $\rho$ and $W_\eff$ belong to $C\big([0,T];B^{s}_{\infty,r}\big)$; using the definition of $W_\eff$ and
the previous property for $u$, this in turn implies that $\rho\in C\big([0,T];B^{s+1}_{\infty,r}\big)$.

This having been established, we see that the analysis of Paragraph \ref{sss:pressure} yields
$\nabla\Pi^0\in C\big([0,T];L^2\cap B^{s}_{\infty,r}\big)$. Recalling the definition of $\nabla\Pi^0$ and the properties established so far,
we finally gather the sough regularity for the pressure gradient, namely $\nabla\Pi\in C\big([0,T];B^{s-2}_{\infty,r}\big)$.

\medbreak
All in all, we have obtained the claimed time regularity properties. This completes the proof of the existence part of Theorem \ref{th:wp-inf}.

\subsection{Stability and uniqueness} \label{ss:stab-uniqueness}

In this subsection, we are going to prove the uniqueness of solutions in the considered functional setting.
No need to say, the Els\"asser formulation \eqref{eq:odd-Els} of the original odd viscosity system will play a fundamental role in our argument.

Notice that, in the previous Subsection \ref{ss:proof-ex}, we have proved that the constructed solution solves not only \eqref{eq:odd}, but also \eqref{eq:odd-Els}. 
However, for the uniqueness proof, this is not enough: we need to establish first that \emph{any} solution of \eqref{eq:odd} also solves \eqref{eq:odd-Els}.
By embeddings, we only need to show this in the lowest regularity framework, corresponding to the case $s=r=1$.

\begin{lemma} \label{l:odd-to-Els}
Let $T>0$ and let the triplet $\big(\rho,u,\nabla\Pi\big)$ be a solution to system \eqref{eq:odd} on $[0,T]\times\R^2$ satisfying the following regularity properties:
\begin{itemize}
 \item $\rho\in L^\infty\big([0,T]\times\R^2\big)\cap C\big([0,T];B^{2}_{\infty,1}\big)$, with $0<\rho_*\leq\rho\leq\rho^*$, for two suitable positive constants
$\rho_*$ and $\rho^*$;
\item $u\in C\big([0,T];L^2\cap B^1_{\infty,1}\big)$;
\item $\nabla\Pi$ belongs to $C\big([0,T];B^{-1}_{\infty,1}\big)$.
\end{itemize}

After denoting by $\o\,:=\,\curl(u)$ the vorticity of the fluid, define $W_\eff\,:=\,u\,-\,2\,\nu_0\,\nabla^\perp\log\rho$ and $\Pi^0\,:=\,\Pi\,-\,\nu_0\,\rho\,\o$.
Then the quadruple $\big(\rho,u,W_\eff,\nabla\Pi^0\big)$ solves the Els\"asser system \eqref{eq:odd-Els}.
\end{lemma}

\begin{proof}
The proof simply consists in following the computations of Subsection \ref{ss:effective} and observing that they are all justified in the considered
functional framework.

As a matter of fact, an inspection of the mass equation reveals that $\d_t\rho$ belongs to $C_T\big(B^1_{\infty,1}\big)$, where, for convenience
of notation, we have set $C_T\big(B^1_{\infty,1}\big)\,:=\,C\big([0,T];B^1_{\infty,1}\big)$.
Then, we gather that $\rho\in C^1([0,T];B^1_{\infty,1})\hookrightarrow C^1_b\big([0,T]\times\R^2\big)$, where, for any $k\in\N$, we denote
$C^k_b\,:=\,C^k\cap W^{k,\infty}$.

From the previous property, we deduce two important facts: first of all that $W_\eff\in C_T(B^1_{\infty,1})$, and secondly that equation \eqref{eq:Dlog}
makes sense in $\mc D'\big([0,T]\times\R^2\big)$. The same can be said about the computations in \eqref{eq:odd-comput}, which in turn imply  \eqref{eq:u^perp}.
The validity of both \eqref{eq:Dlog} and \eqref{eq:u^perp} finally yields \eqref{eq:W^perp}, thus completing the proof of the Lemma.
\end{proof}

With Lemma \ref{l:odd-to-Els} at hand, we can establish a stability estimate for solutions to the Els\"asser system \eqref{eq:odd-Els}. This is a key step in the proof
of uniqueness.
Once again, we can limit ourselves to consider the minimal regularity case $s=r=1$.

\begin{prop} \label{p:stab-Els}
Let $T>0$. Assume to dispose of two sets $\big(\rho^{(1)}, u^{(1)}, W^{(1)}, \nabla\pi_u^{(1)}, \nabla\pi_W^{(1)}\big)$ and
$\big(\rho^{(2)}, u^{(2)}, W^{(2)}, \nabla\pi_u^{(2)}, \nabla\pi_W^{(2)}\big)$ of (weak) solutions to the following system on $[0,T]\times\R^2$:
\begin{equation} \label{eq:gen-Els}
\left\{\begin{array}{l}
\d_t\rho\,+\,u\cdot\nabla\rho\,=\,0 \\[1ex]
\rho\,\d_t u\,+\,\rho\,W\cdot\nabla u\,+\,\nabla\pi_u\,=\,0 \\[1ex]
\rho\,\d_tW\,+\,\rho\,u\cdot\nabla W\,+\,\nabla\pi_W\,=\,0 \\[1ex]
\div u\,=\,\div W\,=\,0\,.
       \end{array}
\right.
\end{equation}
In addition, assume that:
\begin{enumerate}[(a)]
\item the two density functions $\rho^{(1)}$ and $\rho^{(2)}$ both belong to $L^\infty\big([0,T]\times\R^2\big)$;
\item there exist two constants $0<\rho_*\leq \rho^*$ such that both $\rho^{(1)}$ and $\rho^{(2)}$ take value in the interval $[\rho_*,\rho^*]$;
\item all velocity fields $u^{(1)}$, $W^{(1)}$, $u^{(2)}$ and $W^{(2)}$ elong to $L^1\big([0,T];W^{1,\infty}(\R^2)\big)$;
\item the vector fields $\nabla\rho^{(2)}$, $\nabla\pi_u^{(2)}$ and $\nabla\pi_W^{(2)}$ all belong to $L^1\big([0,T];L^\infty(\R^2)\big)$;
\item the products $\left\|u^{(2)}\right\|_{L^\infty}\,\left\|\nabla W^{(2)}\right\|_{L^\infty}$ and
$\left\|W^{(2)}\right\|_{L^\infty}\,\left\|\nabla u^{(2)}\right\|_{L^\infty}$ both belong to $L^1\big([0,T]\big)$;
\item the differences $\rho^{(1)}-\rho^{(2)}$, $u^{(1)}-u^{(2)}$ and $W^{(1)}-W^{(2)}$ all belong to $C^0\big([0,T];L^2(\R^2)\big)$.
\end{enumerate}

Then, there exists a ``universal'' constant $C>0$ and a time-dependent function $\beta\in L^1\big([0,T]\big)$ such that one has the following stability estimate:
\begin{align*}
&\sup_{t\in[0,T]}\left(\left\|\left(\rho^{(1)}-\rho^{(2)}\right)(t)\right\|_{L^2}\,+\,\left\|\left(u^{(1)}-u^{(2)}\right)(t)\right\|_{L^2}\,+\,
\left\|\left(W^{(1)}-W^{(2)}\right)(t)\right\|_{L^2}\right) \\
&\; \leq\,C\,
\left(\left\|\left(\rho^{(1)}-\rho^{(2)}\right)(0)\right\|_{L^2}\,+\,\left\|\left(u^{(1)}-u^{(2)}\right)(0)\right\|_{L^2}\,+\,
\left\|\left(W^{(1)}-W^{(2)}\right)(0)\right\|_{L^2}\right)\, 
e^{C\int^t_0\beta(\t)\,\dd\t}\,.
\end{align*}

\end{prop}

\begin{proof}
For any quantity $f\in\big\{\rho, u, W, \nabla\pi_u, \nabla\pi_W\big\}$, we set $\de f\,:=\,f^{(1)}\,-\,f^{(2)}$. Then, we can write the equations
satisfied by the differences $\big(\de\rho, \de u, \de W\big)$: by simple computations, we find 
\begin{equation} \label{eq:de-sol}
\left\{\begin{array}{l}
\big(\d_t\,+\,u^{(1)}\cdot\nabla\big)\de\rho\,=\,-\,\de u\cdot\nabla\rho^{(2)} \\[1ex]
\rho^{(1)}\,
\big(\d_t\,+\,W^{(1)}\cdot\nabla\big)\de u\,+\,\nabla\de \pi_u\,=\,-\,\de\rho\,\d_tu^{(2)}\,-\,\rho^{(1)}\,\de W\cdot\nabla u^{(2)}\,-\,
\de\rho\,W^{(2)}\cdot\nabla u^{(2)} \\[1ex]
\rho^{(1)}\,
\big(\d_t\,+\,u^{(1)}\cdot\nabla\big)\de W\,+\,\nabla\de \pi_W\,=\,-\,\de\rho\,\d_tW^{(2)}\,-\,\rho^{(1)}\,\de u\cdot\nabla W^{(2)}\,-\,
\de\rho\,u^{(2)}\cdot\nabla W^{(2)}\,.
       \end{array}
\right.
\end{equation}

By classical $L^2$ estimates for the transport equations by divergence-free vector fields, we easily get, for any $t\in[0,T]$, the inequality
\begin{align} \label{est:de-f}
\left\|\big(\de\rho,\de u,\de W\big)(t)\right\|_{L^2}\,\lesssim\,\left\|\big(\de\rho_0,\de u_0,\de W_0\big)\right\|_{L^2}\,+\,
\int^t_0\beta(\t)\,\left\|\big(\de\rho,\de u,\de W\big)(\t)\right\|_{L^2}\dd\t\,,
\end{align}
where we have set $\de f_0\,:=\,f^{(1)}(0)\,-\,f^{(2)}(0)$ for $f\in\big\{\rho, u, W\big\}$, and we have defined
\begin{align*}
\beta(t)\,&:=\,\left\|\nabla\rho^{(2)}(t)\right\|_{L^\infty}\,+\,\left\|\big(\d_tu^{(2)},\d_tW^{(2)}\big)(t)\right\|_{L^\infty} \\
&\qquad\qquad\qquad \,+\,
\left\|\nabla u^{(2)}\right\|_{L^\infty}\,\left(1\,+\,\left\|W^{(2)}\right\|_{L^\infty}\right)\,+\,
\left\|\nabla W^{(2)}\right\|_{L^\infty}\,\left(1\,+\,\left\|u^{(2)}\right\|_{L^\infty}\right)\,.
\end{align*}
Notice that the multiplicative constant appearing in \eqref{est:de-f} depends on the constants $\rho_*$ and $\rho^*$.
At this point, as both densities $\rho^{(1)}$  and $\rho^{(2)}$ stay far from vacuum,
we can use equations \eqref{eq:gen-Els} to bound the time derivatives appearing in the definition of the function $\beta$: this gives us
\[
\left\|\d_tu^{(2)},\d_tW^{(2)}\right\|_{L^\infty}\,\lesssim\,\left\|W^{(2)}\right\|_{L^\infty}\,\left\|\nabla u^{(2)}\right\|_{L^\infty}\,+\,
\left\|u^{(2)}\right\|_{L^\infty}\,\left\|\nabla W^{(2)}\right\|_{L^\infty}\,+\,
\left\|\big(\nabla\pi_u^{(2)},\nabla\pi_W^{(2)}\big)\right\|_{L^\infty}\,,
\]
up again to an implicit multiplicative constant, depending on $\rho_*$ and $\rho^*$.
In particular, it follows from our assumptions that the function $\beta(t)$ belongs to $L^1\big([0,T]\big)$. Thus, we can apply the Gr\"onwall lemma
to inequality \eqref{est:de-f} and get the result.
\end{proof}

In the end, we can prove the uniqueness of solutions in our functional framework.

\begin{proof}[Proof of uniqueness in Theorem \ref{th:wp-inf}]
Let $\big(\rho_0,u_0\big)$ be an initial datum satisfying the assumptions of Theorem \ref{th:wp-inf}. We consider two solutions
$\big(\rho^{(1)},u^{(1)},\nabla\Pi^{(1)}\big)$ and $\big(\rho^{(2)},u^{(2)},\nabla\Pi^{(2)}\big)$ associated to that initial datum,
defined on the same time interval $[0,T]$ and possessing the regularity properties claimed in the statement of the theorem.

For $j=1,2$, let us define 
\[
W_\eff^{(j)}\,:=\,u^{(j)}\,-\,2\,\nu_0\,\nabla^\perp\log\rho^{(j)}\,,\qquad \o^{(j)}\,:=\,\curl\big(u^{(j)}\big)\,,\qquad 
\nabla\Pi^0_{(j)}\,=\,\nabla\left(\Pi^{(j)}\,-\,\nu_0\,\rho^{(j)}\,\o^{(j)}\right)\,.
\]
Owing to Lemma \ref{l:odd-to-Els} above, we know that each quadruple $\left(\rho^{(j)}, u^{(j)}, W_\eff^{(j)}, \nabla\Pi^0_{(j)}\right)$
solves system \eqref{eq:odd-Els}. As that system coincides with \eqref{eq:gen-Els}, with in addition $\nabla\pi_u\,=\,\nabla\pi_W\,=\,\nabla\Pi^0$,
the idea is now to apply the stability estimate of Proposition \ref{p:stab-Els}: let us check that the assumptions of that statement are verified.

First of all, we notice that assumptions (a) and (b) are fulfilled, thanks to condition (i) in Theorem \ref{th:wp-inf}. Next, using the regularity properties
stated in items (i) and (ii) of that theorem and to the embedding $B^1_{\infty,1}\hookrightarrow W^{1,\infty}$, we see that,
for $j=1,2$, $u^{(j)}$ and $W_\eff^{(j)}$ both belong to $L^\infty_T\big(W^{1,\infty}\big)$, thus also assumptions (c) and (e) are checked.
Using item (iii) to study the pressure gradients $\nabla \Pi^0_{(j)}$, we infer that condition (d) is verified.

Thus, it remains us to check assumption (f) of Proposition \ref{p:stab-Els}. To begin with, we observe that
$u^{(j)}\in C_T\big(L^2\big)$ for $j=1,2$, owing to property (ii) in Theorem \ref{th:wp-inf}. Hence, also the difference $\de u\,:=\,u^{(1)}\,-\,u^{(2)}$
belongs to the same space $C_T\big(L^2\big)$.
Next, the initial density being the same,
one has that $\de\rho\,:=\,\rho^{(1)}\,-\,\rho^{(2)}$ verifies the equation
\[
\left(\d_t\,+\,u^{(1)}\cdot\nabla\right)\de\rho\,=\,-\,\de u\cdot\nabla\rho^{(2)}\qquad\qquad
\mbox{ with initial datum }\qquad \de\rho_{|t=0}\,=\,0\,.
\]
Using that $\de u\in C_T\big(L^2\big)$ and $\nabla\rho^{2}\in L^\infty\big([0,T]\times\R^2\big)$, together with classical properties of transport equations,
we discover that $\de\rho\in C_T\big(L^2\big)$ as well.

To conclude the proof, we must establish a similar property for $\de W_\eff\,:=\,W_\eff^{(1)}\,-\,W_\eff^{(2)}$, namely we have to prove
that $\de W_\eff\in C_T\big(L^2\big)$.
For this, we first observe that, by using the third equation in \eqref{eq:odd-Els} and
the divergence-free condition over $W_\eff^{(j)}$, we can write
\[
-\,\div\left(\frac{1}{\rho^{(j)}}\,\nabla\Pi^0_{(j)}\right)\,=\,\div\left(u^{(j)}\cdot\nabla W_\eff^{(j)}\right)\,.
\]
Then, Lemma 2 of \cite{D_2010} and the two properties $u^{(j)}\in C_T\big(L^2\big)$ and $W_\eff^{(j)}\in C_T\big(B^1_{\infty,1}\big)$
imply that $\nabla\Pi^0_{(j)}$ belongs to $C_T\big(L^2\big)$. 
This piece of information yields two important consequences. Firstly, we deduce that
$\nabla\de \Pi^0\,:=\,\nabla\Pi^0_{(1)}\,-\,\nabla\Pi^0_{(2)}$ belongs to the same space $C_T\big(L^2\big)$. In addition, by an inspection of the equation
for $W_\eff^{(2)}$, we deduce that
\[
 \d_tW_\eff^{(2)}\,=\,-\,\frac{1}{\rho^{(2)}}\,u^{(2)}\cdot\nabla W_\eff^{(2)}\,-\,\frac{1}{\rho^{(2)}}\,\nabla\Pi^0_{(2)}\quad \in\,C_T\big(L^2\big)\,\cap\,
 L^\infty\big([0,T]\times\R^2\big)\,.
\]
At this point, we compute the equation for
$\de W_\eff$: by repeating the computations leading to system \eqref{eq:de-sol}, we find
\begin{align*}
\big(\d_t\,+\,u^{(1)}\cdot\nabla\big)\de W_\eff\,=\,f\,,
\end{align*}
where we have defined
\[
 f\,:=\,-\,\frac{1}{\rho^{(1)}}\,\nabla\de\Pi^0\,-\,\frac{1}{\rho^{(1)}}\,\de\rho\,\d_tW_\eff^{(2)}\,-\,\de u\cdot\nabla W_\eff^{(2)}\,-\,
\frac{1}{\rho^{(1)}}\,\de\rho\,u^{(2)}\cdot\nabla W^{(2)}\,.
\]
Owing to the above established properties, we easily see that $f\in C_T\big(L^2\big)$. On the other hand, $\big(\de W_\eff\big)_{|t=0}\,=\,0$ by assumption.
Then, classical results on transport equations by divergence-free vector fields immediately imply that $\de W_\eff$ belongs to
the space $C_T\big(L^2\big)$.
In turn, also assumption (f) of Proposition \ref{p:stab-Els} is fulfilled.

In the end, we have seen that we can apply the stability estimates of Proposition \ref{p:stab-Els} to the Els\"asser system \eqref{eq:odd-Els}.
As the initial data for the two solutions coincide, we immediately get that $\de\rho\equiv0$, $\de u\equiv0$ and $\de W_\eff\equiv0$ in $C_T(L^2)$, so almost
everywhere on $[0,T]\times\R^2$. As the solutions are regular, those relations must hold everywhere on $[0,T]\times\R^2$.
In particular, this implies that $\rho^{(1)}\equiv\rho^{(2)}$ and $u^{(1)}\equiv u^{(2)}$, thus implying the sought uniqueness
of solutions to system \eqref{eq:odd}.
\end{proof}

\subsection{Improved lower bound for the lifespan} \label{ss:lifespan}

In this subsection, we prove the lower bound for the lifespan of the solutions claimed in Theorem \ref{t:lifespan}.
It is worth noticing that this bound does not follow from the classical hyperbolic theory. 
The main ingredient of its proof is the use of improved estimates for solutions to the transport equation in Besov spaces $B^0_{p,r}$ of regularity index $s=0$,
which we now recall.

Assume to have a smooth solution to the initial value problem
\begin{equation}\label{eq:TV}
\begin{cases}
\partial_t f + v \cdot \nabla f = g \\
f_{|t = 0} = f_0\,.
\end{cases}
\end{equation}
Then, restricting to the case $p=+\infty$ for simplicity, one has
\begin{equation}\label{est:transp_B^0}
\| f \|_{L^\infty_T(B^0_{\infty, r})}\, \leq\, C\, \bigg( \| f_0 \|_{B^0_{\infty, r}}\, +\, \| g \|_{L^1_T(B^0_{\infty, r})} \bigg)\;
\left( 1+\int_0^T\| \nabla v(\tau) \|_{L^\infty}{\rm d} \tau \right)\,.
\end{equation}
Thus, the Besov norm $B^0_{p,r}$ of $f$ grows \emph{linearly} in the Lipschitz norm of the transport field $v$, and not exponentially as is the case
in classical transport estimates.
Estimates \eqref{eq:TV} were originally proved by Vishik \cite{Vis}; later, Hmidi and Keraani \cite{HK} proposed a different technique of proof.

\medbreak
For the sake of clarity, we divide our proof in several steps.

\paragraph*{Step 1: low regularity norms.}
To begin with, we observe that, as a consequence of the continuation criterion, the lifespan of a solution does not depend on the regularity of the space in which
it is measured. Thus, fixed an initial datum $\big(\rho_0,u_0\big)$ satisfying the assumptions of Theorem \ref{th:wp-inf},
it is enough to estimate the lifespan of the corresponding solution in the endpoint space $B^1_{\infty,1}$.
This motivates us to define, for $t\geq0$, the (lower order) energy
\[
 \mbb E(t)\,:=\,E_1(t)\,=\,\left\|u(t)\right\|_{L^2}\,+\,\left\|\rho(t)\right\|_{L^\infty}\,+\,\left\|\o(t)\right\|_{B^{0}_{\infty,1}}\,+\,
\left\|\z_\eff(t)\right\|_{B^{0}_{\infty,1}}\,,
\]
with obvious changes in the notation when $t=0$. Here we have followed the notation introduced in Section \ref{s:a-priori} and have denoted by
$\o\,=\,\d_1u_2-\d_2u_1$ and $\z_\eff\,=\,\d_1W_{\eff,2}-\d_2W_{\eff,1}$ the vorticity functions of, respectively, the velocity field $u$ and
the effective velocity $W_\eff\,=\,u-2\nu_0\nabla^\perp\log\rho$.

Observe that the low regularity norms $\|u\|_{L^2}$ and $\|\rho\|_{L^\infty}$ are controlled by \eqref{est:u-L^2} and \eqref{est:rho-inf},
respectively. Then, we must only bound the $B^0_{\infty,1}$ norm of the vorticity functions.
We notice that $\o$ and $\z_\eff$ solve system \eqref{eq:syst-vort}: applying estimate \eqref{est:transp_B^0} to that system,
we find
\begin{align*}
\left\|\big(\o(t),\z_\eff(t)\big)\right\|_{B^0_{\infty,1}} 
\,&\lesssim\,
\left(1\,+\,\int^t_0\Big(\left\|\nabla u\right\|_{L^\infty}+\left\|\nabla W_\eff\right\|_{L^\infty}\Big)\,\dd\t\right) \\
&\qquad \times\Bigg(\left\|\big(\o(0),\z_\eff(0)\big)\right\|_{B^0_{\infty,1}}+ 
\int^t_0\left(\left\|F\right\|_{B^0_{\infty,1}}+\left\|G\right\|_{B^0_{\infty,1}}\right)\dd\t\Bigg)\,,
\end{align*}
where we have defined
\[
 F\,:=\,\nabla^\perp\left(\frac{1}{\rho}\right)\cdot\nabla\Pi^0 \qquad\qquad \mbox{ and }\qquad\qquad
 G\,:=\,\mc B\big(\nabla u,\nabla^2\log\rho\big)\,.
\]
Owing to the estimate
\[
 \left\|\nabla u\right\|_{L^\infty}+\left\|\nabla W_\eff\right\|_{L^\infty}\,\lesssim\,\mbb E\,,
\]
by adding \eqref{est:rho-inf} and \eqref{est:u-L^2} to the previous inequality we easily find
\begin{align} \label{est:life-E_1}
 \E(t)\,\lesssim\,\left(1\,+\,\int^t_0\E(\t)\,\dd\t\right)\,\left(\E(0)\,+\,
\int^t_0\left(\left\|F\right\|_{B^0_{\infty,1}}+\left\|G\right\|_{B^0_{\infty,1}}\right)\dd\t\right)\,.
\end{align}

For estimating the terms $F$ and $G$ in \eqref{est:life-E_1}, we follow the main lines of the analysis of Subsection \ref{ss:est_non-lin}.
Let us start by considering $F$. By taking advantage of \eqref{est:bilin-press} and embeddings, we see that
\[
\left\|F\right\|_{B^0_{\infty,1}}\,\lesssim\,\left\|\nabla\rho\right\|_{B^0_{\infty,1}}\,\left\|\nabla\Pi^0\right\|_{B^1_{\infty,1}}\,.
\]
In order to control the pressure gradient, we use \eqref{est:D-Pi^0-Besov} and \eqref{est:Pi_inf} to infer that
\begin{align*}
\left\|\nabla\Pi^0\right\|_{B^1_{\infty,1}}\,&\lesssim\,
\E\,\Big(1\,+\,\left\|\nabla\rho\right\|_{L^\infty}^4\,+\,\left\|\big(\nabla u,\nabla^2\rho\big)\right\|^2_{L^\infty}\Big)\,\lesssim\,\E\,\Big(1\,+\,\E^2\Big)\,,
\end{align*}
where we have used the interpolation inequality \eqref{est:interp-rho} together with the bound
\begin{align} \label{est:D^2rho-E}
\left\|\nabla^2\rho\right\|_{L^\infty}\,\lesssim\,\left\|\rho\right\|_{L^\infty}\,+\,\left\|\nabla\log\rho\right\|_{B^1_{\infty,1}}\,\lesssim\,
\left\|\rho\right\|_{L^\infty}\,+\,\left\|\big(u,W_\eff\big)\right\|_{B^1_{\infty,1}}\,\lesssim\,\E\,.
\end{align}
The proof of the previous bound is based on a decomposition into low and high frequencies and on the use of Lemma \ref{l:paralin}.
Notice that all the bounds depend implicitly on $\rho_*$, $\rho^*$ and $\|u_0\|_{L^2}$, which are quantities preserved by the flow of the equations. Notice also that,
by simple modifications, one could actually avoid the dependence on the $L^2$ norm of $u_0$.
All in all, we get
\begin{equation} \label{est:F}
\left\|F\right\|_{B^0_{\infty,1}}\,\lesssim\,\left\|\nabla\rho\right\|_{B^0_{\infty,1}}\,\E\,\Big(1\,+\,\E^2\Big)\,.
\end{equation}

Let us now focus on $G$. As already observed in Paragraph \ref{sss:bilinear}, we have
\begin{align*}
 \left\|G\right\|_{B^0_{\infty,1}}\,&\lesssim\,\left\|\nabla u\right\|_{L^\infty}\,\left\|\nabla^\perp\log\rho\right\|_{B^{1}_{\infty,1}}\,+\,
\left\|\nabla\nabla^\perp\log\rho\right\|_{L^\infty}\,\left\|u\right\|_{B^{1}_{\infty,1}}\,.
\end{align*}
This bound, together with Lemma \ref{l:paralin}, immediately implies the following estimate:
\begin{equation} \label{est:G}
  \left\|G\right\|_{B^0_{\infty,1}}\,\lesssim\,\left\|\nabla\rho\right\|_{B^1_{\infty,1}}\,\E\,.
\end{equation}
At this point, we observe that $\nabla\rho$ solves the following transport equation:
\[
 \d_t\nabla\rho\,+\,u\cdot\nabla\nabla\rho\,=\,-\,\sum_{j=1,2}\nabla u_j\,\d_j\rho\,.
\]
Thus, by classical transport estimates in Besov spaces, we gather, for any $t\geq0$, the bound
\begin{equation} \label{est:life_D-rho}
\left\|\nabla\rho(t)\right\|_{B^1_{\infty,1}}\,\lesssim\,\left\|\nabla\rho_0\right\|_{B^1_{\infty,1}}\,
\exp\left(C_0\int^t_0\left\|u\right\|_{B^2_{\infty,1}}\,\dd\t\right)\,.
\end{equation}

Plugging estimates \eqref{est:F} and \eqref{est:G} into \eqref{est:life-E_1} and using \eqref{est:life_D-rho}, in turn we find
\begin{align}
 \label{est:life-E_2}
\E(t)\,&\lesssim\,\left(1\,+\,\int^t_0\E(\t)\,\dd\t\right) \\
\nonumber
&\qquad\times\,\left(\E(0)\,+\,\left\|\nabla\rho_0\right\|_{B^1_{\infty,1}}
\int^t_0\E\,\Big(1\,+\,\E^2\Big)\,\exp\left(C_0\int^\t_0\left\|u\right\|_{B^2_{\infty,1}}\,\dd\s\right)\dd\t\right)\,.
\end{align}
It is apparent that the previous estimate presents a loss of one derivative, as, for controling
$\E\approx \|u\|_{B^1_{\infty,1}}\,+\,\|W_\eff\|_{B^1_{\infty,1}}$, we need an information over the higher order norm $\|u\|_{B^2_{\infty,1}}$.
A similar situation was handled in \cite{Cobb-F}: the basic idea is to make use of the continuation criterion in order to control
the higher order norm in terms of the lower order one.

\paragraph*{Step 2: bounding the high regularity norms.}
More precisley, let us introduce the higher order energy of the solution as
\[
 \H(t)\,:=\,\left\|u(t)\right\|_{L^2}\,+\,\left\|\rho(t)\right\|_{L^\infty}\,+\,\left\|\o(t)\right\|_{B^{1}_{\infty,1}}\,+\,
\left\|\z_\eff(t)\right\|_{B^{1}_{\infty,1}}\,.
\]
Then, as a consequence of \eqref{est:E_second} we have
\[
 \H(t)\,\lesssim\,\H(0)\,+\,\int^t_0A(\t)\,\H(\t)\,\dd\t\,.
\]
Now, resorting to inequality \eqref{est:D^2rho-E}, it is easy to see that
\[
 A(t)\,\lesssim\,\E(t)\,\left(1+\E(t)^{5/2}\right)\,\lesssim\,\E(t)\,\left(1+\E(t)^3\right)\,.
\]
for an implicit multiplicative constant which may depend also on $\rho_*$.
This in particular implies the estimate
\begin{equation} \label{est:H}
 \H(t)\,\lesssim\,\H(0)\,\exp\left(C_1\int^t_0\E(\t)\,\left(1\,+\,\E(\t)^3\right)\,\dd\t\right)\,.
\end{equation}


\paragraph*{Step 3: uniform bounds in a time interval $[0,T]$.}
With inequalities \eqref{est:life-E_2} and \eqref{est:H} at hand, we can conclude our argument.
We define the time $T>0$, up to which bounding uniformly the energy $\E$, as
\[
T\,:=\,\sup\left\{t>0\,\Big|\; \left\|\nabla\rho_0\right\|_{B^1_{\infty,1}}
\int^t_0\E(\t)\,\Big(1+\E(\t)^3\Big)\,\exp\left(C_0\int^\t_0\H(\s)\dd\s\right)\dd\t\leq\E(0)\,e^{C_0\H(0)t} \right\}.
\]
For convenience of notation, we also define the following quantities:
\[
\psi(t)\,:=\,e^{C_0\H(0)t}\qquad \mbox{ and }\qquad \Psi(t)\,:=\,\int^t_0\psi(\t)\,\dd\t\,=\,\frac{1}{C_0\,\H(0)}\,\left(e^{C_0\H(0)t}-1\right)\,.
\]

Now, it follows from \eqref{est:life-E_2} and the definition of $T$ that, for any $t\in[0,T]$, one has
\[
 \E(t)\,\leq\,C_2\,\E(0)\,\psi(t)\,\left(1\,+\,\int^t_0\E(\t)\,\dd\t\right)\,.
\]
Thus, an application of the Gr\"onwall lemma yields, for any $t\in[0,T]$, the bounds
\begin{align} \label{est:int-E}
\int^t_0\E(\t)\,\dd\t\,\leq\,e^{C_2\,\E(0)\,\Psi(t)}\,-\,1\qquad\qquad \mbox{ and }\qquad\qquad
\E(t)\,\leq\,C_2\,\E(0)\,\psi(t)\,e^{C_2\,\E(0)\,\Psi(t)}\,.
\end{align}
Without loss of generality, we can assume $C_2\geq1$. Notice also that $\psi(t)\geq1$ for any time $t\geq0$.
Hence, using the bounds in \eqref{est:int-E} into \eqref{est:H}, we discover that, for any $t\in[0,T]$, one has
\begin{align*}
\H(t)\,&\lesssim\,\H(0)\,\exp\left(C_3\,\E(0)\,\left(1+\E(0)^3\right)\,\int^t_0\psi^4(\t)\,e^{4C_2\E(0)\Psi(\t)}\,\dd\t\right)\,,
\end{align*}
where we have defined $C_3\,:=\,C_1\,(C_2)^3$.
Recalling the definition of the functions $\psi(t)$ and $\Psi(t)$,  we deduce in particular that
\begin{align*}
\H(t)\,&\lesssim\,\H(0)\exp\left(C_3\,\E(0)\,\left(1+\E(0)^3\right)\,\int^t_0e^{4C_0\H(0)\t}\,e^{4\frac{C_2\E(0)}{C_0\H(0)}\big(\exp(C_0\H(0)\t)-1\big)}\,\dd\t\right) \\
&\lesssim\,\H(0)\exp\left(K\,\H(0)\,\left(1+\E(0)^3\right)\,\int^t_0e^{K\H(0)\t}\,e^{K\big(\exp(K\H(0)\t)-1\big)}\,\dd\t\right)\,,
\end{align*}
where we have used the fact that $\E(0)\leq \H(0)$ and we have defined $K$ as the maximum between the constants $C_3$, $4C_0$ and $4C_2/C_0$.
Observing that
\begin{align*}
K\,\H(0)\,\int^t_0e^{K\H(0)\t}\,e^{K\big(\exp(K\H(0)\t)-1\big)}\,\dd\t\,&=\,
\frac{1}{K}\,\int^t_0\frac{\dd}{\dt}\left(e^{K\big(\exp(K\H(0)\t)-1\big)}\right)\,\dd\t \\
&=\,\frac{1}{K}\,\left(e^{K\big(\exp(K\H(0)\t)-1\big)}\,-\,1\right)\,,
\end{align*}
we finally conclude that, for any $t\in[0,T]$, one has
\begin{align}
\label{est:H_T}
\H(t)\,\leq\,C_4\,\H(0)\,e^{\mc E_0\,\Theta(t)}\,, 
\end{align}
where we have defined
\[
\mc E_0\,:=\, \frac{\left(1+\E(0)^3\right)}{K}\qquad\quad \mbox{ and }\qquad\quad 
\Theta(t)\,:=\,e^{K\big(\exp(K\H(0)t)-1\big)}\,-\,1\,.
\]

\paragraph*{Step 4: consequences of the uniform bounds.}
At this point, we use inequalities \eqref{est:int-E} and \eqref{est:H_T} to bound the integral appearing in the definition of the time $T$.
More precisely, 
using also that $\psi(t)$ is an increasing function of time, with $\psi(t)\geq1$ for all $t\geq0$, we see that, for any $t\in[0,T]$, we can estimate
\begin{align*}
&\int^t_0\E(\t)\,\Big(1\,+\,\E(\t)^3\Big)\,\exp\left(C_0\int^\t_0\H(\s)\,\dd\s\right)\dd\t \\
&\leq\,
C_5\,\E(0)\,\left(1+\E(0)^3\right)\,\psi(t)\,\int^t_0\psi^3(\t)\,e^{4C_2\E(0)\Psi(\t)}\,
\exp\left(C_0\,C_4\,\H(0)\int^\t_0e^{\mc E_0\,\Theta(\s)}\,\dd\s\right)\,\dd\t \\
&\leq\,
C_5\,K\,\mc E_0\,\E(0)\,\psi(t)\,\int^t_0e^{3C_0\H(0)\t}\,e^{4\frac{C_2\E(0)}{C_0\H(0)}\big(\exp(C_0\H(0)\t)-1\big)}\,
\exp\left(C_0\,C_4\,\H(0)\int^\t_0e^{\mc E_0\,\Theta(\s)}\,\dd\s\right)\,\dd\t\,,
\end{align*}
where, in passing from the first to the second inequality, we have simply used the definitions of the constant $K_0$ and of the functions $\psi(t)$ and $\Psi(t)$.
Now, if we set $K_1$ to be the maximum between the values $C_5\,K$, $3C_0$, $4C_2/C_0$ and $C_0C_4$, we can bound
the last line of the previous series of inequalities by
\begin{align} \label{eq:bigger}
K_1\,\mc E_0\,\H(0)\,\psi(t)\int^t_0e^{K_1\H(0)\t}\,e^{K_1\big(\exp(K_1\H(0)\t)-1\big)}\,\exp\left(K_1\,\H(0)\int^\t_0e^{\mc E_0\,\Theta(\s)}\,\dd\s\right)\,\dd\t\,.
\end{align}
Without loss of generality, we can assume that $K_1\geq1$. Then, we remark that the inequality
\[
e^{K_1\H(0)\t}\,\leq\,e^{K_1\big(\exp(K_1\H(0)\t)-1\big)}
\]
holds true for any time $\t\geq0$. Using this observation twice, we see that
\begin{align*}
 e^{K_1\H(0)\t}\,e^{K_1\big(\exp(K_1\H(0)\t)-1\big)}\,&\leq\,e^{2\,K_1\big(\exp(K_1\H(0)\t)-1\big)} \\
 &\leq\,e^{K_1\big(e^{2\,K_1\big(\exp(K_1\H(0)\t)-1\big)}-1\big)}\,.
\end{align*}
Inserting this estimate into \eqref{eq:bigger}, in turn we obtain, for any $t\in[0,T]$, the bound
\begin{align}
\label{est:integral-T}
 &\int^t_0\E(\t)\,\Big(1\,+\,\E(\t)^3\Big)\,\exp\left(C_0\int^\t_0\H(\s)\,\dd\s\right)\dd\t \\
\nonumber
&\leq\,K_1\,\mc E_0\,\H(0)\,\psi(t)\int^t_0e^{K_1\big(e^{2K_1\big(\exp(K_1\H(0)\t)-1\big)}-1\big)}\,
\exp\left(K_1\,\H(0)\int^\t_0e^{\mc E_0\,\Theta(\s)}\,\dd\s\right)\,\dd\t \\
\nonumber
&\leq\,K_1\,\mc E_0\,\H(0)\,\psi(t)\int^t_0e^{K_1\,\mc E_0\big(e^{2K_1\big(\exp(K_1\H(0)\t)-1\big)}-1\big)}\,
\exp\left(K_1\,\H(0)\int^\t_0e^{K_1\,\mc E_0\,\Theta_1(\s)}\,\dd\s\right)\,\dd\t\,,
\end{align}
where, in passing from the first to the second inequality, we have used the fact that $\mc E_0\geq1$ and where
$\Theta_1(t)$ is defined as
\[
\Theta_1(t)\,:=\,e^{2K_1\big(\exp(K_1\H(0)t)-1\big)}\,-\,1\,.
\]

Notice that, with those definitions, we have the equality
\[
e^{K_1\,\mc E_0\big(e^{2K_1\big(\exp(K_1\H(0)\t)-1\big)}-1\big)}\,=\,e^{K_1\,\mc E_0\,\Theta_1(\t)} \,,
\]
which in particular implies that
\begin{align*}
&K_1\,\H(0)\,e^{K_1\,\mc E_0\big(e^{2K_1\big(\exp(K_1\H(0)t)-1\big)}-1\big)}\,\exp\left(K_1\,\H(0)\int^t_0e^{K_1\,\mc E_0\,\Theta_1(\t)}\,\dd\t\right) \\
&\qquad\qquad\qquad\qquad\qquad\qquad\qquad\qquad\qquad\qquad
=\,\frac{\dd}{\dt}\exp\left(K_1\,\H(0)\int^t_0e^{K_1\,\mc E_0\,\Theta_1(\t)}\,\dd\t\right)\,.
\end{align*}
Inserting this relation into \eqref{est:integral-T}, we can explicitly compute the integral in the right-hand side and get
the following estimate:
\begin{align*}
&\int^t_0\E(\t)\,\Big(1\,+\,\E(\t)^3\Big)\,\exp\left(C_0\int^\t_0\H(\s)\,\dd\s\right)\dd\t \\
&\qquad\qquad\qquad\qquad\qquad\qquad\qquad
\leq\,\mc E_0\,\psi(t)\,\bigg(\exp\left(K_1\,\H(0)\int^t_0e^{K_1\,\mc E_0\,\Theta_1(\t)}\,\dd\t\right)\,-\,1\bigg)\,.
\end{align*}
Notice that, thanks to the definition of $\Theta_1(t)$, the integral inside the exponential term can be greatly simplified.
Indeed, we observe that
\[
\frac{\dd}{\dt}\Theta_1(t)\,=\,2\,(K_1)^2\,\H(0)\,e^{K_1\H(0)t}\,e^{2K_1\big(\exp(K_1\H(0)t)-1\big)}\,.
\]
Now, as the exponential terms are greater than or equal to $1$, we can estimate
\begin{align*}
\int^t_0e^{K_1\,\mc E_0\,\Theta_1(\t)}\,\dd\t\,\leq\,
\int^t_0e^{K_1\H(0)\t}\,e^{2K_1\big(\exp(K_1\H(0)\t)-1\big)}e^{K_1\,\mc E_0\,\Theta_1(\t)}\,\dd\t\,,
\end{align*}
from which we deduce the bound
\begin{align*}
K_1\,\H(0)\int^t_0e^{K_1\,\mc E_0\,\Theta_1(\t)}\,\dd\t\,\leq\,\frac{1}{2\,(K_1)^2\,\mc E_0}
\int^t_0\frac{\dd}{\dt}e^{K_1\,\mc E_0\,\Theta_1(\t)}\,\dd\t\,\leq\,e^{K_1\,\mc E_0\,\Theta_1(t)}-1\,,
\end{align*}
where we have also used the fact that $2(K_1)^2\mc E_0\geq1$.
In turn, this implies that, for all $t\in[0,T]$, one must have the inequality
\begin{align}
\label{est:integral-T2}
&\int^t_0\E(\t)\,\Big(1\,+\,\E(\t)^3\Big)\,\exp\left(C_0\int^\t_0\H(\s)\,\dd\s\right)\dd\t\,
\leq\,\mc E_0\,\psi(t)\,\bigg(\exp\left(e^{K_1\,\mc E_0\,\Theta_1(t)}-1\right)\,-\,1\bigg)\,.
\end{align}

\paragraph*{Step 5: estimate of the time $T$ (end of the proof).}
After the work carried out in the previous step, we are now in the position of giving an explicit lower bound for the time $T$, thus completing
the proof.

First of all, we remark that, by definition of $T$ (see Step 3) and continuity of all the quantities, at time $t=T$ we must have the equality
\[
\left\|\nabla\rho_0\right\|_{B^1_{\infty,1}}
\int^T_0\E(t)\,\Big(1+\E(t)^3\Big)\,\exp\left(C_0\int^t_0\H(\t)\dd\t\right)\dd t\,=\,\E(0)\,e^{C_0\H(0)T}\,.
\]
On the other hand, in Step 4 we have obtained the bound \eqref{est:integral-T2}: using also the definition of $\mc E_0\,=\,(1+\E(0)^3)/K$
(where $K$ is a universal constant) and the fact that $1\leq\psi(t)\leq \psi(T)$ for all $t\geq0$, that inequality in turn implies
\begin{align*}
\left\|\nabla\rho_0\right\|_{B^1_{\infty,1}}\,\frac{1+\E(0)^3}{K}\,\psi(T)\,\bigg(\exp\left(e^{K_1\,\frac{1+\E(0)^3}{K}\,\Theta_1(T)}-1\right)\,-\,1\bigg)\,\geq\,
\E(0)\,e^{C_0\H(0)T}\,. 
\end{align*}
As $\psi(T)\,=\,e^{C_0\H(0)T}$, from the previous bound we gather
\[
\left\|\nabla\rho_0\right\|_{B^1_{\infty,1}}\,\frac{1+\E(0)^3}{K}\,\bigg(\exp\left(e^{K_1\,\frac{1+\E(0)^3}{K}\,\Theta_1(T)}-1\right)\,-\,1\bigg)\,\geq\,
\E(0)\,.
\]
Of course, this estimate is satisfied as soon as $T>0$ verifies the more restrictive condition
\[
\left\|\nabla\rho_0\right\|_{B^1_{\infty,1}}\,\frac{1+\E(0)^3}{K}\,\bigg(\exp\left(e^{K_1\,\frac{1+\E(0)^3}{K}\,\Theta_1(T)}-1\right)\,-\,1\bigg)\,\geq\,
1+\E(0)\,. 
\]
Hence, we infer that there exists a suitable multiplicative constant $K_2>0$, depending
on the quantities $\rho_*$ and $\rho^*$ and on $\|u_0\|_{L^2}$ but not on the higher regularity norms of the initial datum, such that
\begin{equation} \label{est:expon}
e^{K_2\,(1+\E(0)^3)\,\Theta_1(T)}-1\,\geq\,\log\left(1\,+\,\frac{1}{K_2\,\left(1+\E(0)^2\right)\,\left\|\nabla\rho_0\right\|_{B^1_{\infty,1}}}\right)\,.
\end{equation}
Next we observe that, as the multiplicative constants are allowed to depend on the triplet $\big(\rho_*,\rho^*\,\|u_0\|_{L^2}\big)$, one has
\[
\E(0)\,\approx\,1\,+\,\left\|\left(\o_0,\Delta\rho_0\right)\right\|_{B^0_{\infty,1}}\qquad \mbox{ and }\qquad 
\H(0)\,\approx\,1\,+\,\left\|\left(\o_0,\Delta\rho_0\right)\right\|_{B^1_{\infty,1}}\,.
\]
Then, keeping in mind the definition of the function $\Theta_1(t)$, from \eqref{est:expon}
we can finally deduce the claimed lower bound for $T$. This completes the proof of Theorem \ref{t:lifespan}.


\appendix

\section{Appendix -- A primer on Littlewood-Paley theory} \label{app:LP}

In this appendix, we collect basic results from Fourier analysis and Littlewood-Paley theory which we have used in our study.
The material presented here is rather classical; unless otherwise specified, we refer to Chapters 2 and 3 of \cite{BCD} for a complete presentation.

\subsection{Littlewood-Paley decomposition and Besov spaces} \label{ss:LP}

We start by introducing the so-called Littlewood-Paley decomposition. 
For simplicity of exposition, we deal with the $\R^d$ case, with $d\geq1$, the only relevant case for us.

The Littlewood-Paley decomposition is based on a non-homogeneous dyadic partition of unity with
respect to the Fourier variable.  For this, 
we fix a smooth radial function $\chi$ supported in the ball $B(0,2)$ of center $0$ and radius $2$, equal to $1$ in a neighborhood of $B(0,1)$
and such that $r\mapsto\chi(r\,e)$ is nonincreasing over $\R_+$ for all unitary vectors $e\in\R^d$. Set
$\varphi\left(\xi\right)=\chi\left(\xi\right)-\chi\left(2\xi\right)$ and
$\vphi_j(\xi):=\vphi(2^{-j}\xi)$ for all $j\geq0$.
The dyadic blocks $(\Delta_j)_{j\in\Z}$ are defined\footnote{Throughout this work, we agree  that  $f(D)$ stands for 
the pseudo-differential operator $u\mapsto\mc{F}^{-1}[f(\xi)\,\what u(\xi)]$.}  by
$$
\Delta_j\,:=\,0\quad\mbox{ if }\; j\leq-2,\qquad\Delta_{-1}\,:=\,\chi(D)\qquad\mbox{ and }\qquad
\Delta_j\,:=\,\varphi(2^{-j}D)\quad \mbox{ if }\;  j\geq0\,.
$$
We  also introduce the following low frequency cut-off operator:
\begin{equation} \label{eq:S_j}
S_ju\,:=\,\chi(2^{-j}D)\,=\,\sum_{k\leq j-1}\Delta_{k}\qquad\mbox{ for }\qquad j\geq0\,.
\end{equation}
Note that $S_j$ is a convolution operator. More precisely, if we denote $\mc F(f)\,=\,\what f$ the Fourier transform of a function $f$ and $\mc F^{-1}$
the inverse Fourier transform, after defining
$$
K_0\,:=\,\mc F^{-1}\chi\qquad\qquad\mbox{ and }\qquad\qquad K_j(x)\,:=\,\mathcal{F}^{-1}[\chi (2^{-j}\cdot)] (x) = 2^{jd}K_0(2^j x)\,,
$$
we have, for all $j\in\N$ and all tempered distributions $u\in\mc S'$, the formula $S_ju\,=\,K_j\,*\,u$.
As a consequence, the $L^1$ norm of $K_j$ is independent of $j\geq0$, thus $S_j$ maps continuously $L^p$ into itself, for any $1 \leq p \leq +\infty$.
A similar remark applies also to the operators $\Delta_j$, for any $j\geq -1$.

Then, one can prove the \emph{Littlewood-Paley decomposition} of tempered distributions:
\[
\forall\,u\in\mc S'\,,\qquad\qquad  u\,=\,\sum_{j\geq-1}\Delta_ju\quad \mbox{ in the sense of }\quad \mc S'\,.
\]
One of the advantages of using Littlewood-Paley decomposition is that one disposes of very simple relations showing how derivatives
act on spectrally localized functions. This is expressed by the so-called \emph{Bernstein inequalities}, which are contained in the next lemma.
  \begin{lemma} \label{l:bern}
Let  $0<r<R$.   A constant $C>0$ exists so that, for any nonnegative integer $k$, any couple $(p,q)$ 
in $[1,+\infty]^2$, with  $p\leq q$,  and any function $u\in L^p$,  we  have, for all $\lambda>0$,
\begin{align*}
&{\Supp}\, \widehat u\, \subset\,   B(0,\lambda R)\,=\,\big\{\xi\in\R^d\,\big|\;|\xi|\leq\lambda R \big\}\qquad
\Longrightarrow\qquad
\|\nabla^k u\|_{L^q}\, \leq\,
 C^{k+1}\,\lambda^{k+d\left(\frac{1}{p}-\frac{1}{q}\right)}\,\|u\|_{L^p} \\[1ex]
&{\Supp}\, \widehat u   \, \subset\, \big\{\xi\in\R^d\,\big|\, r\lambda\leq|\xi|\leq R\lambda\big\} \\
&\qquad\qquad\qquad\qquad\qquad\qquad\qquad \Longrightarrow\qquad C^{-k-1}\,\lambda^k\|u\|_{L^p}\,\leq\,
\|\nabla^k u\|_{L^p}\,
\leq\,C^{k+1} \, \lambda^k\|u\|_{L^p}\,.
\end{align*}
\end{lemma}   

By use of the Littlewood-Paley decomposition, we can define the class of (non-homogeneous) Besov spaces on $\R^d$.
\begin{defi} \label{d:B}
  Let $s\in\R$ and $1\leq p,r\leq+\infty$. The \emph{non-homogeneous Besov space}
$B^{s}_{p,r}\,=\,B^s_{p,r}(\R^d)$ is defined as the subset of tempered distributions $u\in\mc S'$ for which
$$
\|u\|_{B^{s}_{p,r}}\,:=\,
\left\|\left(2^{js}\,\|\Delta_ju\|_{L^p}\right)_{j\geq -1}\right\|_{\ell^r}\,<\,+\infty\,.
$$
\end{defi}
Besov spaces are interpolation spaces between Sobolev spaces. In fact, for any $k\in\N$ and~$p\in[1,+\infty]$,
we have the following chain of continuous embeddings:
$$
B^k_{p,1}\hookrightarrow W^{k,p}\hookrightarrow B^k_{p,\infty}\,,
$$
where  $W^{k,p}$ stands for the classical Sobolev space of $L^p$ functions with all the derivatives up to the order $k$ in $L^p$.
When $1<p<+\infty$, we can refine the previous result (this is the non-homogeneous version of Theorems 2.40 and 2.41 in \cite{BCD}): we have
$B^k_{p, \min (p, 2)}\hookrightarrow W^{k,p}\hookrightarrow B^k_{p, \max(p, 2)}$.
In particular, for all $s\in\R$, we deduce the equivalence $B^s_{2,2}\equiv H^s$, with equivalence of norms.
The class of Besov spaces also includes the classical H\"older spaces: for any $k\geq0$ and $\alpha\in\,]0,1[\,$, one has the equivalence
$C^{k,\alpha}\,\equiv\,B^{k+\alpha}_{\infty,\infty}$, with equivalence of norms.

As an immediate consequence of the first Bernstein inequality, one gets the following result, which generalises Sobolev embeddings
to the framework of Besov spaces.
\begin{prop}\label{p:embed}
The space $B^{s_1}_{p_1,r_1}$ is continuously embedded in the space $B^{s_2}_{p_2,r_2}$ for all indices satisfying $p_1\,\leq\,p_2$ and
$$
s_2\,<\,s_1-d\left(\frac{1}{p_1}-\frac{1}{p_2}\right)\qquad\qquad\mbox{ or }\qquad\qquad
s_2\,=\,s_1-d\left(\frac{1}{p_1}-\frac{1}{p_2}\right)\;\;\mbox{ and }\;\;r_1\,\leq\,r_2\,. 
$$
\end{prop}

In particular, we get the following chain of continuous embeddings: provided that the triplet
$(s, p, r) \in \mathbb{R} \times [1, +\infty]^2$ satisfies the condition
\begin{equation}\label{eq:AnnLInfty}
s > \frac{d}{p} \qquad\qquad \text{ or } \qquad\qquad s = \frac{d}{p}\quad \text{ and }\quad r = 1\,,
\end{equation}
then we have
\begin{equation*}
B^s_{p,r} \hookrightarrow B^{s - \frac{d}{p}}_{\infty, r} \hookrightarrow B^0_{\infty, 1} \hookrightarrow L^\infty\,.
\end{equation*}
Likewise, under either the condition $s>1+d/p$ and $r\in[1,+\infty]$, or the condition $s=1+d/p$ and $r=1$, on has 
embedding $B^s_{p,r}\hookrightarrow W^{1,\infty}$.
In particular, condition \eqref{eq:Lip-cond} implies $B^s_{\infty,r}\hookrightarrow W^{1,\infty}$.

\subsection{Some elements of paradifferential calculus}\label{ss:NHPC}

We now presents some basic results in paradifferential calculus. Once again, we refer to Chapter 2 of \cite{BCD} for full details.

First of all, we introduce the paraproduct decomposition (after J.-M. Bony, see \cite{Bony}). 
Constructing the paraproduct operator relies on the observation that, 
formally, the product  of two tempered distributions $u$ and $v$ may be decomposed into 
\begin{equation*} 
u\,v\;=\;\mathcal{T}_u(v)\,+\,\mathcal{T}_v(u)\,+\,\mathcal{R}(u,v)\,,
\end{equation*}
where we have defined
$$
\mathcal{T}_u(v)\,:=\,\sum_jS_{j-1}u\,\Delta_j v\qquad\qquad\mbox{ and }\qquad\qquad
\mathcal{R}(u,v)\,:=\,\sum_j\sum_{|j'-j|\leq1}\Delta_j u\,\Delta_{j'}v\,.
$$
The above operator $\mc T$ is called ``paraproduct'' whereas
$\mc R$ is called ``remainder''.
The paraproduct and remainder operators have many nice continuity properties, some of them are collected in the next statement.
\begin{prop}\label{p:op}
For any $(s,p,r)\in\R\times[1,+\infty]^2$ and $\s>0$, the paraproduct operator 
$\mathcal{T}$ maps continuously $L^\infty\times B^s_{p,r}$ in $B^s_{p,r}$ and  $B^{-\s}_{\infty,\infty}\times B^s_{p,r}$ in $B^{s-\s}_{p,r}$.
Moreover, the following estimates hold:
$$
\|\mathcal{T}_u(v)\|_{B^s_{p,r}}\,\lesssim\,\|u\|_{L^\infty}\,\|\nabla v\|_{B^{s-1}_{p,r}}\qquad\mbox{ and }\qquad
\|\mathcal{T}_u(v)\|_{B^{s-\s}_{p,r}}\,\lesssim\,\|u\|_{B^{-\s}_{\infty,\infty}}\,\|\nabla v\|_{B^{s-1}_{p,r}}\,.
$$
For any $(s_1,p_1,r_1)$ and $(s_2,p_2,r_2)$ in $\R\times[1,+\infty]^2$ such that 
$s_1+s_2>0$, $1/p:=1/p_1+1/p_2\leq1$ and~$1/r:=1/r_1+1/r_2\leq1$,
the remainder operator $\mathcal{R}$ maps continuously~$B^{s_1}_{p_1,r_1}\times B^{s_2}_{p_2,r_2}$ into~$B^{s_1+s_2}_{p,r}$.
In the case $s_1+s_2=0$, provided $r=1$, the operator $\mathcal{R}$ is continuous from $B^{s_1}_{p_1,r_1}\times B^{s_2}_{p_2,r_2}$ with values
in $B^{0}_{p,\infty}$.
\end{prop}

The previous continuity properties have not been directly used in the paper, but they are a tool to prove the so-called \emph{tame estimates}, which
we are now going to state.
From now one, we specialise on the case $p=+\infty$, as this is the only case of interest for our study.

\begin{cor}\label{c:tame}
Let $(s, r)\in\R\times[1,+\infty]$ be such that that $s > 0$. Then, we have
\begin{equation*}
\forall\, f, g \in B^s_{\infty,r}\,, \quad\qquad \| fg \|_{B^s_{\infty,r}}\, \lesssim \,\| f \|_{L^\infty}\, \|g\|_{B^s_{\infty,r}}\, +\,
\| f \|_{B^s_{\infty,r}} \,\| g \|_{L^\infty}\,.
\end{equation*}
\end{cor}

Notice that, as a consequence of Proposition \ref{p:op} and relation \eqref{eq:AnnLInfty}, the spaces $B^s_{\infty,r}$ are Banach algebras as long as
$s > 0$.
On the contrary, the space $B^0_{\infty, 1}$ \emph{fails} to be an algebra. As a matter of fact, if $f, g \in B^0_{\infty, 1}$, 
we can use Proposition \ref{p:op} to bound the paraproducts $\mc T_f(g)$ and $\mc T_g(f)$, but the
remainder $\mathcal{R}(f, g)$ is not controlled in $B^0_{\infty,1}$.

\medbreak
Next, we switch our attention to paralinearisation results. Many are the statements available in the literature; in the present paper,
we have only needed the following lemma, borrowed from \cite{D-F} (see Proposition 3 therein).
\begin{lemma} \label{l:paralin}
Fix an open interval $I\subset\R$ and let $F:I\to\R$ be a smooth function. Then, for all compact subinterval $J\subset I$,
any $s>0$ and any $r\in[1,+\infty]$, there exists a constant $C>0$ such that the following fact holds true:
for all functions $a$ taking values in $J$ and whose gradient satisfies $\nabla a\in B^{s-1}_{\infty,r}$, one has
$\nabla\big(F\circ a\big)\,\in\,B^{s-1}_{\infty,r}$ , together with the estimate
\[
 \left\|\nabla\big(F\circ a\big)\right\|_{B^{s-1}_{\infty,r}}\,\leq\,C\,\left\|\nabla a\right\|_{B^{s-1}_{\infty,r}}\,.
\]
\end{lemma}

Finally, we consider commutator estimates.
We only focus on the commutator between the frequency localisation operator $\Delta_j$ and a trasport operator $\d_t+v\cdot\nabla$.
We refer to Lemma 2.100 and Remark 2.101 of \cite{BCD} for a proof.
Notice that estimate \eqref{eq:lCommBLinfty} below is not contained in \cite{BCD}, but it easily follows by slight
modifications to the arguments of the proof (see in particular the control of the term $R^3_j$ at pages 113-114 of \cite{BCD}).


\begin{lemma}\label{l:CommBCD}
Let $(s,r)\in\R\times[1,+\infty]$ satisfy condition \eqref{eq:Lip-cond}.
Take a vector field $v\in B^s_{\infty,r}$ and denote by
$\big[ v \cdot \nabla, \Delta_j \big]\,=\,(v \cdot \nabla) \Delta_j - \Delta_j (v \cdot \nabla)$ the commutator between the transport operator $v\cdot\nabla$ and the frequency localisation operator $\Delta_j$.

Then, the following estimates hold true:
\begin{equation*}
\forall\, f \in B^s_{\infty,r}\,, \qquad\qquad  2^{js}\left\| \big[ v \cdot \nabla, \Delta_j \big] f  \right\|_{L^\infty}\,\lesssim\,
c_j \Big( \|\nabla v \|_{L^\infty} \| f \|_{B^s_{\infty,r}}\, +\, \|\nabla v \|_{B^{s-1}_{\infty, r}} \|\nabla f \|_{L^\infty} \Big)\,,
\end{equation*}
and also
\begin{equation}\label{eq:lCommBLinfty}
\forall\, g \in B^{s-1}_{\infty, r}\,, \qquad\qquad
2^{j(s-1)} \left\| \big[ v \cdot \nabla, \Delta_j \big] g  \right\|_{L^\infty}\, \lesssim\, d_j \Big( \|\nabla v \|_{L^\infty} \| g \|_{B^{s-1}_{\infty, r}} +
\|\nabla v \|_{B^{s-1}_{\infty, r}} \| g \|_{L^\infty} \Big)\,,
\end{equation}
where the sequences $\big(c_j\big)_{j\geq -1}$ and $\big(d_j\big)_{j\geq -1}$ are sequences belonging to the unit ball of $\ell^r$. 
\end{lemma}




\end{document}